\documentclass[11pt]{article}
\usepackage[utf8]{inputenc}  
\usepackage[british]{babel}
\usepackage{amsmath,amssymb,amsfonts,amsthm}
\usepackage{tikz}
\usepackage{hyperref}

\usepackage[mono=false]{libertine}
\usepackage[T1]{fontenc}
\usepackage[cmintegrals,libertine]{newtxmath}
\usepackage[cal=euler, calscaled=.95, frak=euler]{mathalfa}
\useosf
\usepackage[a4paper,vmargin={3.5cm,3.5cm},hmargin={2.5cm,2.5cm}]{geometry}
\usepackage[font=sf, labelfont={sf,bf}, margin=1cm]{caption}
\usepackage{numprint}

\theoremstyle{plain}
\newtheorem{theorem}{Theorem}
\newtheorem{lemma}[theorem]{Lemma}
\newtheorem{proposition}[theorem]{Proposition}
\newtheorem{corollary}[theorem]{Corollary}

\newtheorem{conjecture}[theorem]{Conjecture}

\theoremstyle{definition}

\theoremstyle{remark}
\newtheorem{remark}[theorem]{Remark}

\newcommand{\N}{\mathbb{N}}

\newcommand{\R}{\mathbb{R}}
\renewcommand{\P}{\mathbb{P}}
\newcommand{\E}{\mathbb{E}}
\newcommand{\e}{\operatorname{e}} 
\renewcommand{\d}{\mathop{}\!\mathrm{d}}
\newcommand{\map}{\mathfrak{m}}
\newcommand{\Map}{\mathrm{M}}
\newcommand{\Triv}{\mathrm{T}}
\newcommand{\Tree}{\mathfrak{T}}
\newcommand{\Mapb}{\mathbb{M}}
\newcommand{\face}{\mathrm{f}}
\newcommand{\defect}{\mathrm{d}}
\newcommand{\gen}{\mathrm{g}}
\newcommand{\ver}{\mathrm{v}}
\newcommand{\spar}{\mathrm{s}}
\DeclareMathOperator{\Core}{Core}

\DeclareMathOperator{\Ker}{Ker}
\DeclareMathOperator{\Defect}{Defect}
\newcommand{\kernel}{\mathfrak{K}}

\newcommand{\cv}[1][n]{\enskip\mathop{\longrightarrow}^{}_{#1 \to \infty}\enskip}
\newcommand{\cvloi}[1][n]{\enskip\mathop{\longrightarrow}^{(d)}_{#1 \to \infty}\enskip}

\newcommand{\cvproba}[1][n]{\enskip\mathop{\longrightarrow}^{\P}_{#1 \to \infty}\enskip}

\newcommand{\equi}[1][n]{\enskip\mathop{\sim}^{}_{#1 \to \infty}\enskip}

\usepackage{mathtools}

\DeclarePairedDelimiter\floor{\lfloor}{\rfloor}

\let\originalleft\left
\let\originalright\right
\renewcommand{\left}{\mathopen{}\mathclose\bgroup\originalleft}
\renewcommand{\right}{\aftergroup\egroup\originalright}

\DeclareSymbolFont{extraup}{U}{zavm}{m}{n}
\DeclareMathSymbol{\vardspade}{\mathalpha}{extraup}{81}
\DeclareMathSymbol{\varheart}{\mathalpha}{extraup}{86}
\DeclareMathSymbol{\vardiamond}{\mathalpha}{extraup}{87}
\DeclareMathSymbol{\varclub}{\mathalpha}{extraup}{84}

\makeatletter
\renewcommand*{\@fnsymbol}[1]{\ensuremath{\ifcase#1\or  \vardspade \or \varheart \or \vardiamond\or \varclub \or
   \mathsection\or \mathparagraph\or \|\or **\or \dagger\dagger   \or \ddagger\ddagger \else\@ctrerr\fi}}
\makeatother

\author{
Nicolas \textsc{Curien}\thanks{Universit\'e Paris-Saclay.\hfill  \href{mailto:nicolas.curien@gmail.com}{\texttt{nicolas.curien@gmail.com}}}
\qquad\&\qquad
Igor \textsc{Kortchemski}\thanks{CNRS \& CMAP, \'{E}cole polytechnique.\hfill  \href{mailto:igor.kortchemski@math.cnrs.fr}{\texttt{igor.kortchemski@math.cnrs.fr}}}
\qquad\&\qquad
Cyril \textsc{Marzouk}\thanks{CMAP, \'{E}cole polytechnique.\hfill  \href{mailto:cyril.marzouk@polytechnique.edu}{\texttt{cyril.marzouk@polytechnique.edu}}}
}

\title{The mesoscopic geometry of sparse random maps}

\begin{document}
\date{}
\maketitle

\begin{abstract}
We investigate the structure of large uniform random maps with $n$ edges, $\face_n$ faces, and with genus $\gen_n$ in the so-called sparse case, where the ratio between the number vertices and edges tends to $1$. We focus on two regimes: the planar case $(\face_n, 2\gen_n) = (\spar_n, 0)$ and the unicellular case with moderate genus $(\face_n, 2 \gen_n) = (1, \spar_n-1)$, both when $1 \ll \spar_n \ll n$. Albeit different at first sight, these two models can be treated in a unified way using a probabilistic version of the classical core--kernel decomposition. In particular, we show that the number of edges of the core of such maps, obtained by iteratively removing degree $1$ vertices, is concentrated around $\sqrt{n \spar_{n}}$. Further, their kernel, obtained by contracting the vertices of the core with degree $2$, is such that the sum of the degree of its vertices exceeds that of a trivalent map by a term of order $\sqrt{\spar_{n}^{3}/n}$; in particular they are trivalent with high probability when $\spar_{n} \ll n^{1/3}$. This enables us to identify a mesoscopic scale $\sqrt{n/\spar_n}$ at which the scaling limits of these random maps can be seen  as the local limit of their kernels, which is the dual of the UIPT in the planar case and the infinite three-regular tree in the unicellular case, where each edge is replaced by an independent (biased) Brownian tree with two marked points.
\end{abstract}

\begin{figure}[!ht]\centering
\includegraphics[width=0.55\linewidth]{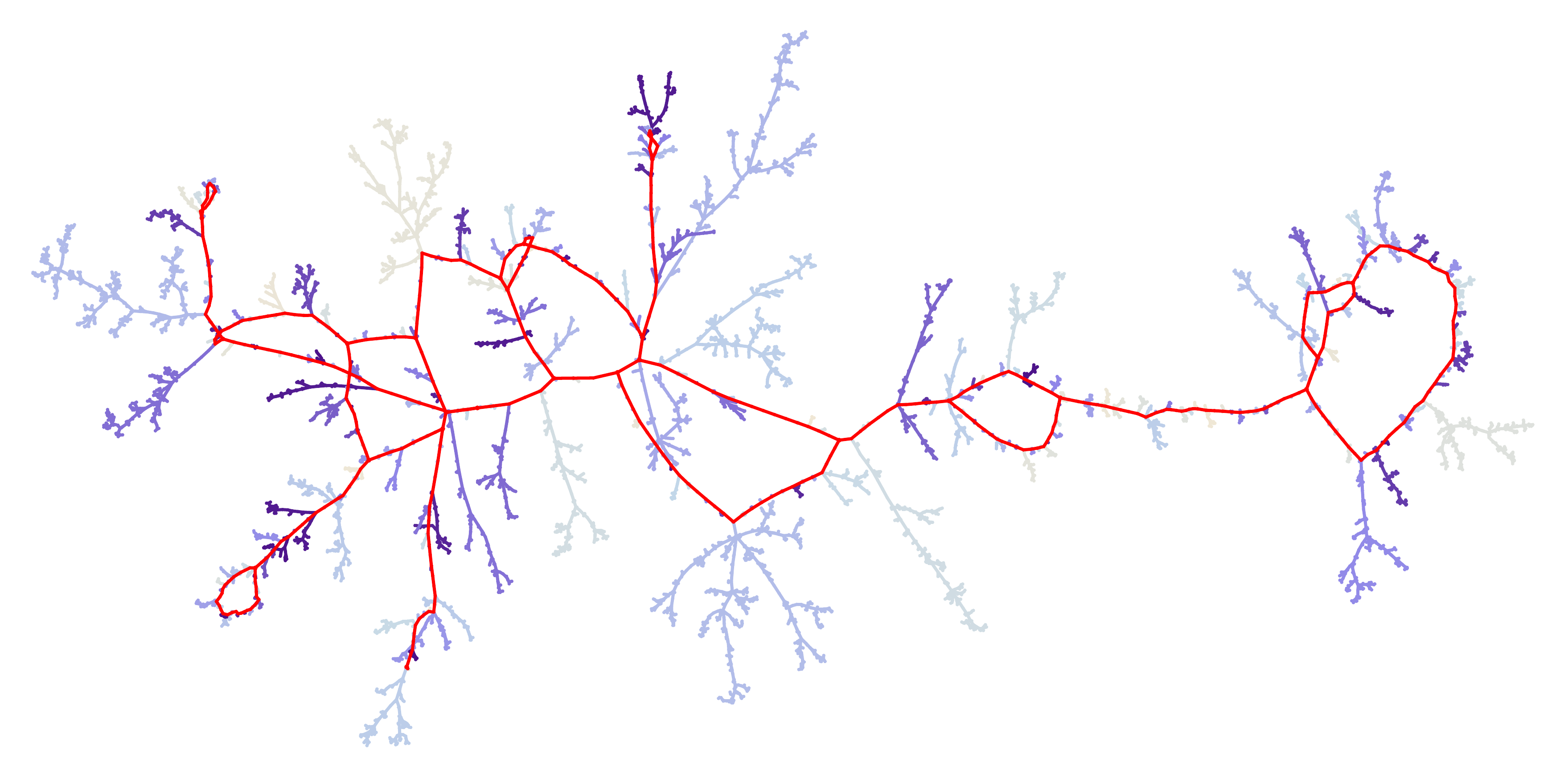}%
\includegraphics[width=0.45\linewidth]{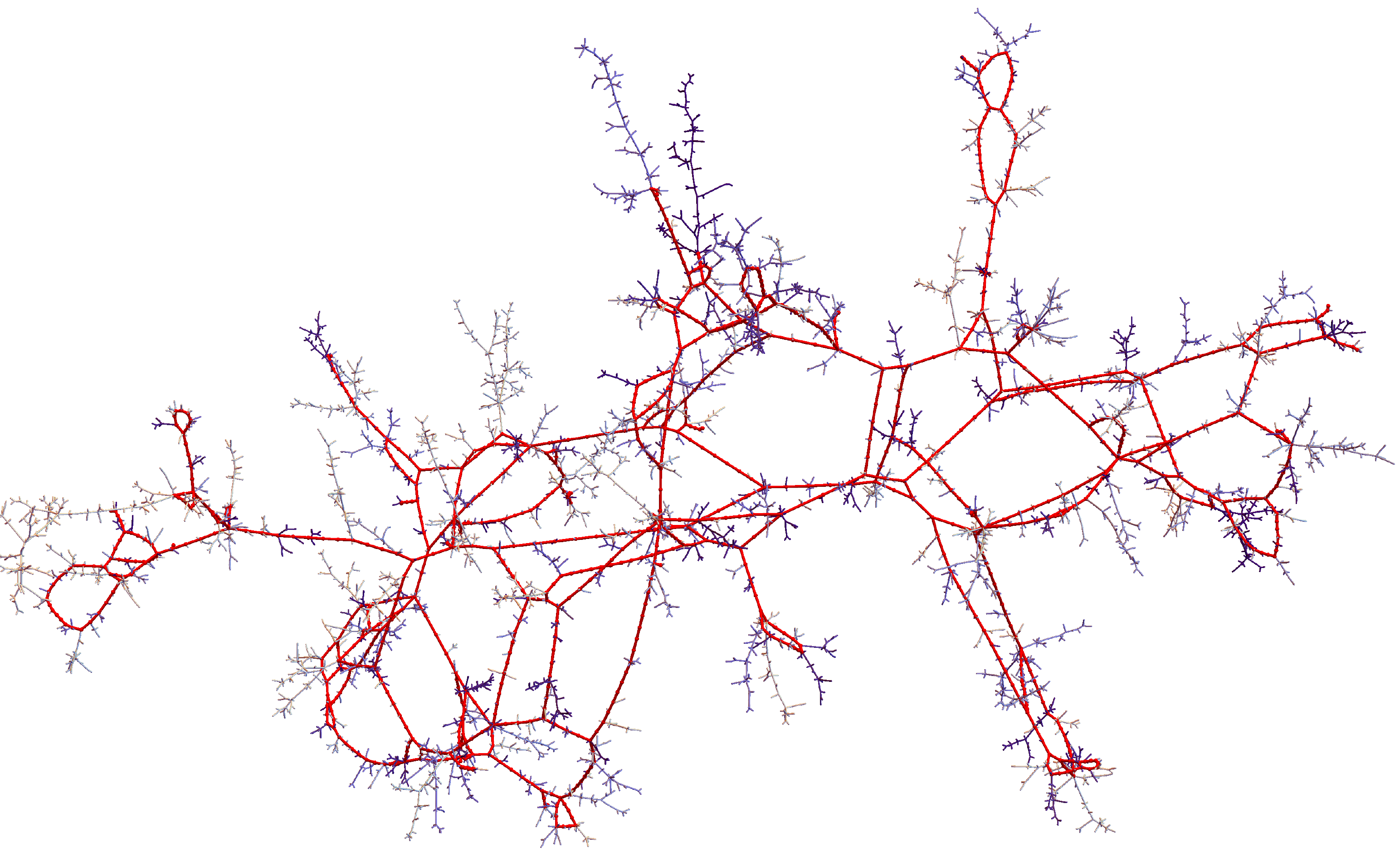}
\caption{Simulations of random plane maps with $n=\numprint{10000}$ edges and $\face_{n}$ faces, with respectively $\face_{n}=\lfloor n^{0.3}\rfloor$ and $\face_{n}=\lfloor n^{0.5} \rfloor$. Their core is represented in red.
}
\label{fig:scales_intro}
\end{figure}

\clearpage
\tableofcontents

\section{Introduction}

\subsection{Random maps with prescribed Euler-parameters}

Sampling uniform random maps with a prescribed number of edges, faces, and genus (by Euler's  formula, this also fixes the numbers of vertices) is a convenient way to probe different random geometries under various topological constraints. After the deep and intensive works devoted to the study of large random plane maps and the Brownian sphere,  recently much attention has been  devoted to the study of (classes of) maps all fixed ``Euler-parameters''.
See e.g.~\cite{FG14,KM21b} for plane maps and~\cite{BL21,BL20,Ray15,ACCR13,Lou21,JL21b} for high genus maps. 

We shall denote by $   \mathfrak{M}_n{(\face,\gen)}$ the set of all (rooted, non necessarily bipartite) maps with $n$ edges, $ \face$ faces, and genus $ \gen$, and by $ \Map_n{( \face,\gen)}$ a map chosen uniformly at random in this set. In this paper, we propose to study large random maps in the so-called \textbf{sparse} regime, where the ratio between the number vertices and edges tends to $1$. Precisely, by Euler's formula the map $ \Map_n( \face_n, \gen_n)$ has  $n+2- \spar_n$ vertices with $\spar_n = \face_n + 2 \gen_n$, quantity which will be called below the \emph{sparsity parameter}, and the sparse regime consists in $\spar_n = o(n)$.
Although we shall not treat here this model in full generality, the big picture we uncover is that such random maps look like uniform almost trivalent maps with  $\face_n$ faces in genus $ \gen_n$, and where each edge is replaced by a bipointed plane tree of size of order $n/\spar_n$.

Specifically, in the present work, we shall be interested in the two ``extreme'' cases  namely the \textbf{planar} case $ \gen_n=0$ and the \textbf{unicellular} case $ \face_n = 1$. We shall fix in the rest of this paper a sequence of integers $\spar_n$ so that
\begin{equation}\label{eq:standing}
\spar_n \cv \infty \qquad \text{and} \qquad  \frac{\spar_n}{n}   \cv 0,
\end{equation} 
and investigate the geometry of $ \Map_n{(\spar_n,0)}$ and $ \Map_n{(1, (\spar_n-1)/2)}$, which both have the same sparsity parameter $\spar_n = \face_n + 2 \gen_n$. 
Obviously, we implicitly restrict ourselves to odd integers $\spar_n$ when considering the second case. 
Let us first review the literature about those models.
\[\begin{array}{|c|c|c|}
\hline
& \mbox{Planar case}~\cite{FG14,KM21b} & \mbox{Unicellular case}~\cite{JL21b} \\
\hline
\mbox{Genus} &  \mathbf{0} & (\spar_n-1)/2 \\
\#\mbox{Faces} & \spar_n &  \mathbf{1} \\
\#\mbox{Edges} & n & n \\
\mbox{Uniform map} & \Map_n(\spar_n,0) & \displaystyle \Map_n(1, (\spar_n-1)/2)\\
\hline
\end{array}\]

\paragraph{Planar case.} Recently, Fusy and Guitter~\cite{FG14}  were interested in two- and three-point functions of biconditioned planar maps $ \Map_n(\spar_n,0)$ and have predicted that outside the so-called ``pure gravity'' class, typical distances in uniform planar maps with $n$ edges and $n^\alpha$ faces, with $\alpha \in (0,1)$, are of order $n^{(2-\alpha)/4}$. This has been recently confirmed in~\cite{KM21b} in the case of \emph{bipartite} planar maps.
Precisely, it is shown there that the scaling limit of such maps, after scaling distances by $n^{(2- \alpha)/4}$, is the celebrated Brownian sphere, which was first proved to be the limit of large uniform random quadrangulations~\cite{LG13, Mie13}, and then of many different discrete models of planar maps, as in~\cite{ABA21,BJM14,NR18,CLG19,Mar19}
and many other papers. Let us mention that~\cite{KM21b} actually deals with the more general model of Boltzmann maps, with face weights. This was proved by combining a classical bijective encoding of bipartite maps via labelled trees and the criterion from~\cite{Mar19} with new local limit estimates for random walks.

\paragraph{Unicellular case.} Very recently, Janson \& Louf~\cite{JL21b} have been interested in the geometry of uniform unicellular maps with moderate high genus, i.e.~$\Map_n( 1, (\spar_{n}-1)/2)$ with $\spar_n$ satisfying~\eqref{eq:standing}. Their main result is  that, after rescaling by $ \sqrt{n/\spar_n}$, the distribution of the sequence of the lengths of the shortest cycles in the map asymptotically matches that of the shortest non-contractible loops in Weil--Peteresson random surfaces in high genus, which are both given by an inhomogeneous Poisson process with explicit intensity, see Section~\ref{sec:comments}. Let us mention that unicellular maps $ \Map_n(1, \gen_{n})$ whose genus  is proportional to the number of edges
 have also been investigated~\cite{ACCR13,JL21a,Ray15}.  They also form a toy model of hyperbolic geometry.
\medskip

In this work, we investigate the combinatorial structure as well as the geometry of $\Map_n(\spar_{n},0)$ and $\Map_n( 1, (\spar_{n}-1)/2)$ at the mesoscopic scale $ \sqrt{n/ \spar_n}$. As opposed to the works cited above, we rely here on a different approach based on the core--kernel decomposition, diffracted through a probabilistic lens. This casts a new light on the above results, see Section~\ref{sec:comments}. Let us first review these decompositions.

\subsection{Core--Kernel decompositions of maps}
Without further notice, all maps considered in this work will be finite and rooted, i.e.~with a distinguished oriented edge.  
If $ \map$ is a map, we write $\mathrm{Vertice}(\map)$, $\mathrm{Edges}(\map)$, and $\mathrm{Face}(\map)$ for its set of vertices, edges, and faces respectively. 
Let us  recall the concepts of core and kernel, which are instrumental in the classical theory of random graphs, see e.g.~\cite{JKLP93,Luc91,NRR15,NR18,Lou21,CMS09,Cha10}. Starting from $ \map$ and repeatedly removing vertices of degree $1$, we obtain a map $ \Core( \map)$, called the \textbf{core} of $\map$.
We then replace all maximal paths of vertices of degree $2$ in this core by single edges to get another map $ \Ker( \map)$,  called the \textbf{kernel} of $\map$,  which only has vertices of degree at least $3$. When $ \map$ is not a tree, the core and its kernel are nonempty.  The root edge is canonically transferred from $ \map$ to $ \Core( \map)$ and then to $  \Ker( \map)$, see Figure~\ref{fig:decomposition} and Figure~\ref{fig:decomp-torus}. Notice that the three maps $\map$, $ \Core( \map)$, and $ \Ker(\map)$ all have the same number of faces and the same genus. 

\begin{figure}[!ht]\centering
\includegraphics[width=1\linewidth]{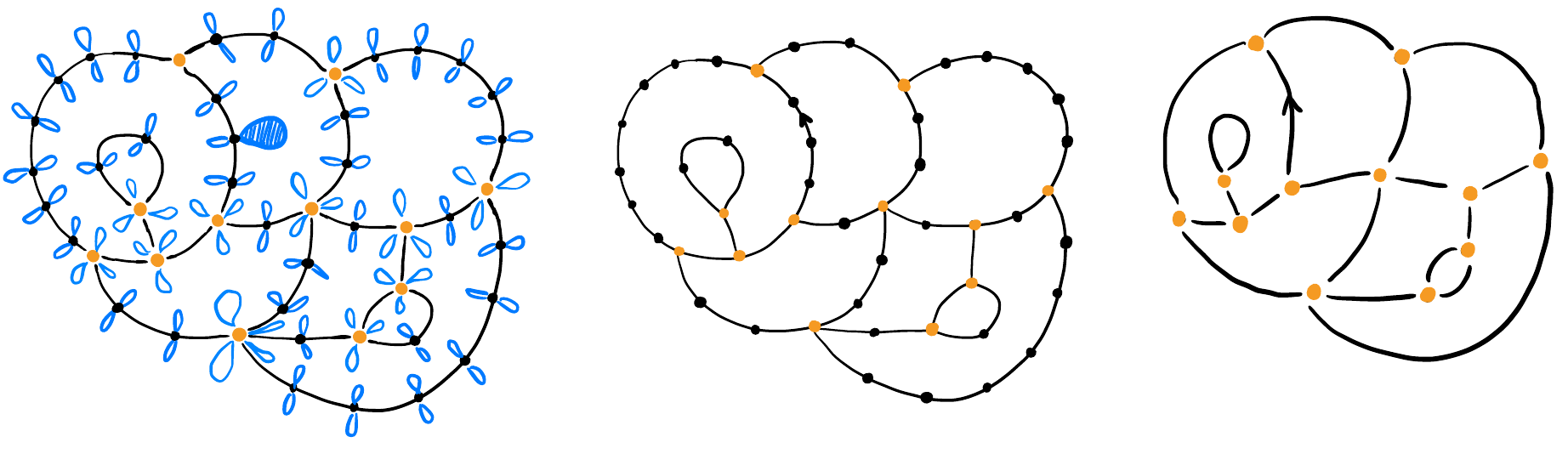}
\caption{Decomposition of a planar map (on the left) into its core after  iteratively removing degree one vertices (on the middle), and then into its kernel by contracting vertices of degree $2$ (on the right). The root edge $ \vec{e}$ of the map is carried by the thick blue tree on the left hand side or on the first edge of the core in clockwise order around the root face. It is transferred to the core and the kernel in a natural way: the root edge $ \vec{e}_c$ of the core is $\vec{e}_c= \vec{e}$ if it already belongs to the core, otherwise it is carried by the tree grafted to the right of the origin of $\vec{e}_c$. }
\label{fig:decomposition}
\end{figure}

\begin{figure}[!ht]
 \begin{center}
 \includegraphics[width=1\linewidth]{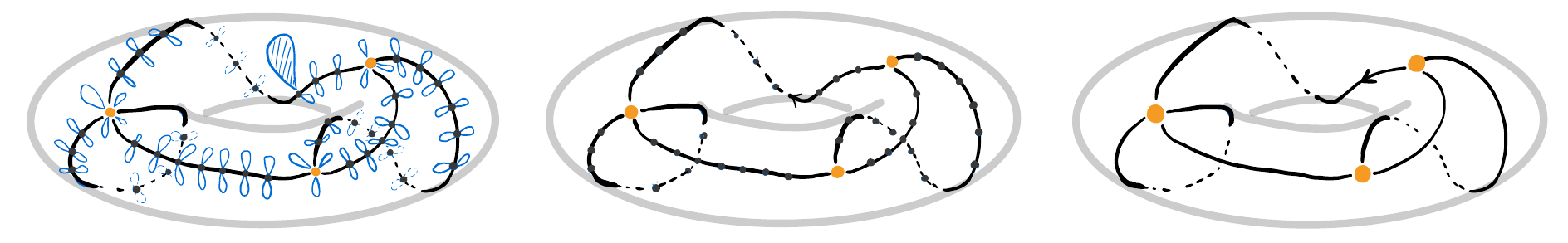}
 \caption{Decomposition of a map of the torus (left) into its core (middle) and kernel (right). Here also, the root edge is carried by the thick blue tree and is transferred naturally to its core and kernel. \label{fig:decomp-torus}}
 \end{center}
 \end{figure}

In the above decomposition, the kernel is a map with only vertices of degree at least $3$. If  $ \mathfrak{K}$ is such a map, then we denote by $\Defect( \mathfrak{K}) \ge 0$ the number defined by
\[\Defect( \mathfrak{K}) = \sum_{v \in \mathrm{Vertices}(  \mathfrak{K})} (\deg(v)-3)
=  2\#\mathrm{Edges}(\mathfrak{K}) - 3\#\mathrm{Vertices}(\mathfrak{K}).\]
We call this number the \textbf{defect number} of $ \mathfrak{K}$; it quantifies how far $ \mathfrak{K}$ is from being \emph{trivalent}, which corresponds to the case $\Defect( \mathfrak{K}) = 0$. 
For $\face \geq 1, \gen \geq 0$, and $\defect \geq 0$, we  denote by $\mathcal{T}_{\defect}( \face, \gen)$ the set of all rooted maps with $ \face$ faces, in genus $ \gen$, whose vertices all have  degree at least $3$, and which have defect number $\defect$.  Let us note that Euler's formula yields for maps in $\mathcal{T}_{\defect}( \face, \gen)$:
\begin{equation}\label{eq:Euler_defauts}
\#\mathrm{Edges} = 3\#\mathrm{Faces} - \Defect + 6(\mathrm{Genus}-1),
\end{equation}
so controlling the defect number is equivalent to controlling the number of edges (and thus of vertices). It turns that  the laws of the core and kernel of a uniform random map with fixed number of edges, faces, and genus are explicit, see Proposition~\ref{prop:loi_core_kernel}. This will be instrumental to all our results. It is interesting to note that the core--kernel decompositions have been used to study enumerative properties of large unicellular maps with fixed genus~\cite{CMS09,Cha10}. Here, we pursue a  probabilistic approach in a more general context.

\paragraph{Volumes of the core and of the kernel.} 
Recall our standing assumptions~\eqref{eq:standing} and denote by $ \Mapb_{n}^{\spar_{n}}$ either $ \Map_{n}(\spar_{n},0)$ or $ \Map_{n}(1,(\spar_{n}-1)/2)$. Our first main results, Theorems~\ref{thm:defauts_ker} and~\ref{thm:aretes_core}, describe the asymptotic behaviour of the size of the core and kernel of $ \Mapb_{n}^{\spar_{n}}$. In particular, we identify a phase transition according to whether the sparsity parameter $\spar_{n}$ is at most of the same order as $n^{1/3}$, or is much larger. 
A model-dependent constant $\lambda_\circ$ appears in these results:
\begin{equation}\label{eq:lambda}
\lambda_\circ = 1-\frac{\sqrt{3}}{2} \quad\text{in the planar case }\enskip 
\qquad\text{and}\qquad
\lambda_\circ = 0 \quad\text{in the unicellular case}.
\end{equation}
It turns out to be the density of loop-edges in the local limit of the kernel of $ \Mapb_n^{s_n}$, see Section~\ref{sec:defauts_bis}.

Throughout this work, we use the notation $X_n \displaystyle\mathop{\to}^\P X$ and $X_n \displaystyle\mathop{\to}^{(d)} X$ to refer  to respectively  convergence in probability and in distribution of a sequence of random variables $X_n$ to a limit $X$. By abuse of notation, we shall also write $X_n \displaystyle\mathop{\to}^{(d)} \mu$ when $\mu$ is a probability measure to refer to the weak convergence of the law of $X_n$ to $\mu$.
We denote by $ \mathrm{Poi}(c)$ the Poisson law with mean $c \ge 0$, which is interpreted as the Dirac mass at $0$ when $c=0$.

\begin{theorem}[Defect number of the kernel] 
\label{thm:defauts_ker}
Assume that $\spar_n$ satisfies~\eqref{eq:standing} and let $\lambda_\circ$ as in~\eqref{eq:lambda}.
\begin{enumerate}
\item\label{thm:defauts_ker-Poisson} If $n^{-1/3} \spar_n \to a$ for some $a \in [0,\infty)$, then
\[\Defect(\Ker( \Mapb_n^{\spar_{n}})) \cvloi \mathrm{Poi} \left(3 (1-\lambda_{\circ}) \sqrt{\frac{3}{2} a^3} \right).\]
In particular  if $n^{-1/3} \spar_n \to  0$, then the probability that $\Ker(\Mapb_n^{\spar_{n}})$ is trivalent tends to $1$ as $n \to \infty$.

\item\label{thm:defauts_ker-Gaussien} If $n^{-1/3} \spar_n \to \infty$, then
\[\sqrt{\frac{n}{\spar_n^{3}}} \cdot \Defect(\Ker(\Mapb_n^{\spar_{n}})) \cvproba 3 (1-\lambda_{\circ}) \sqrt{\frac{3}{2}}.\]
\end{enumerate}
\end{theorem}

\begin{remark}
Theorem~\ref{thm:defauts_ker} establishes a phase transition in the appearance of vertices of degree larger than or equal to $4$ in the kernel of $ \Mapb_n^{\spar_n}$ at  order $ \spar_n \approx n^{1/3}$. This is consistent with~\cite[Lemma~3]{Cha10}, which shows that for fixed $\gen$ the kernel of $ \Map_{n}(1,\gen)$ is trivalent with probability tending to $1$ as $n \rightarrow \infty$. We suspect similar phases transitions to occur at orders $ \spar_n \approx n^{1- 2/(3k)}$ for the appearance of vertices of degree larger than or equal to $3+k$ for $k = 2,3, \dots$ and perhaps similar Poisson statistics for the number of such vertices when $ \spar_n \sim \text{Cst } n^{1- 2/(3k)}$.
\end{remark}

Since the kernel of a map has the same genus and number of  faces as the original map, by Euler's formula~\eqref{eq:Euler_defauts}, Theorem~\ref{thm:defauts_ker} also provides the asymptotic behaviour of the number of edges (and therefore of vertices) of the kernel  of $ \Map_{n}(\spar_{n},0)$ or $ \Map_{n}(1, (\spar_{n}-1)/2)$. 
The main tool to leverage the explicit laws of the core and kernel of a uniform random maps with fixed number of edges, faces and, genus in order to obtain these limit theorems is a  careful analysis of a so-called \emph{contraction operation}, which allows to iteratively create defects from a trivalent map; see Section~\ref{sec:contraction}.

\begin{theorem}[Number of edges in the core] 
\label{thm:aretes_core}
Assume that $\spar_n$ satisfies~\eqref{eq:standing}.
Then
\[\frac{1}{\sqrt{n\spar_n}} \cdot \#\mathrm{Edges}(\Core( \Mapb_n^{\spar_{n}})) \cvproba \sqrt{\frac{3}{2}}.\]
\end{theorem}

It is interesting to note that for uniform random plane maps with $n$ edges, the number of edges of the kernel and of the core concentrate around $(4-4 \sqrt{6}/3)n$ and $\sqrt{6}n/3$ respectively, with Gaussian fluctuations of order $\sqrt{n}$ in both cases, see~\cite[Theorem~5]{NR18}. In this direction, we shall establish (Corollary \ref{cor:TCL_core}) a Central Limit Theorem for  $\#\mathrm{Edges}(\Core( \Mapb_n^{\spar_{n}}))$ conditionally given the number of edges of $\Ker( \Mapb_n^{\spar_{n}})$.
This is sufficient to deduce an unconditioned CLT for the number of edges of the core when $\spar_n = O(\sqrt{n})$, but we believe this is true in general, see precisely Conjecture~\ref{conjecture_TCL}.

As a side result of independent interest, we obtain explicit asymptotic enumeration estimates when $\spar_n= O(n^{1/3})$. Let us mention that such estimates for $\# \mathfrak{M}_{n}(  \face_{n},0)$ when $\face_{n}/n \rightarrow f \in (0,\infty)$ and of $\# \mathfrak{M}_{n}(  1,\gen_{n})$ when $\gen_{n}/n \rightarrow g \in (0,\infty)$ are given respectively in~\cite[Theorem~1]{BCR93} and~\cite[Theorem~3]{ACCR13}. In the sparse regime, to the best of our knowledge the following ones are new.

\begin{corollary}
\label{cor:enumeration}
If $\gen=0$ and $\face_n \to \infty$ with $n^{-1/3} \face_n \to f \in [0,\infty)$, then
\[\# \mathfrak{M}_n(\face_n,0)
\equi \frac{\exp(-(2-\sqrt{3}) (3f/2)^{3/2})}{4 \pi} \cdot n^{-3} \cdot 4^{n}  \cdot \left( 2^{1/3} \frac{\e n}{\face_{n}}\right)^{3\face_{n}/2} 
.\]
On the other hand, if $\face_n=1$ and $\gen_n \to \infty$ with $n^{-1/3} \gen_n \to g \in [0,\infty)$, then
\[\# \mathfrak{M}_n(1,\gen_n)
\equi \frac{1}{2\pi } \cdot \gen_n^{-1/2} \cdot n^{-3/2} \cdot 4^{n}  \cdot \left(  \frac{\e n^3}{12 \gen_n}\right)^{\gen_n}
.\]
\end{corollary}

\subsection{The intermediate scales in biconditioned maps}

In another direction, we are interested in the asymptotic geometry of $ \Mapb_n^{\spar_{n}}$.  First, in addition to our first results which describe the size of the core and kernel, it will also become clear that conditionally on these parameters, they are uniformly distributed. In particular, when $\spar_n = o(n^{1/3})$, by Theorem~\ref{thm:defauts_ker} the kernel is with high probability a uniformly chosen trivalent map, either with $ \spar_{n}$ faces in the planar case, or with genus $ (\spar_{n}-1)/2$ in the unicellular case. The local limits of those objects are well known:
\begin{itemize} 
\item In the planar case, by~\cite{Ste18} uniform trivalent plane maps converge in distribution in the local topology to the dual map of the Uniform Infinite Planar Triangulation (\textrm{UIPT}) of type 1. This result has recently been extended to the case of essentially trivalent plane maps (i.e.~when the defect number is negligible compared to the size of the map) by Budzinski~\cite{Bud21}.
\item In the unicellular case, we shall prove in Section~\ref{ssec:unicellularestimates} using the configuration model that an essentially trivalent unicellular map in high genus converges locally towards the three-regular tree.
\end{itemize}
Second, let us explain how to reconstruct the original map from its kernel (see also~\cite[Section~3.1]{Cha10}). A chain of vertices with degree two in the core and the trees grafted on it can be seen as a single plane tree with two distinguished  vertices. Note that some care is needed at the vertices with degree three and higher so that a tree gets assigned to a unique chain of edges of the core. Our convention is that such a tree is grafted on the chain immediately to its right when turning around the corner. Therefore one can directly construct the map from its kernel by replacing each edge of the kernel by a tree with two distinguished ordered vertices. The bipointed tree replacing the root edge of the kernel additionally carries an oriented edge on its right part which is the root of the entire map. See Figure~\ref{fig:splitarbre} for an illustration. 

\begin{figure}[!ht]\centering
\includegraphics[width=1\linewidth]{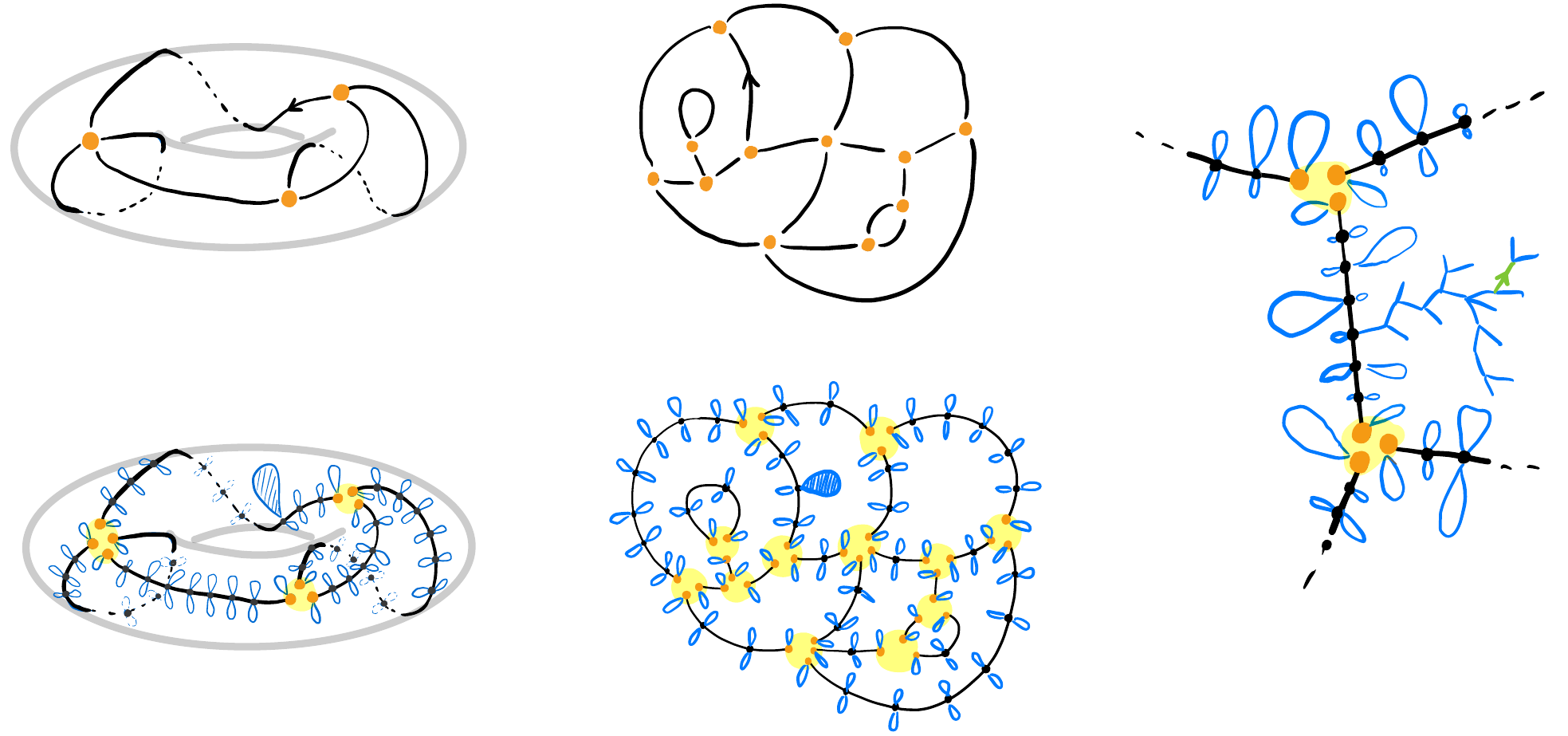}
\caption{Reconstructing a map from its kernel by replacing edges by bipointed trees.}
\label{fig:splitarbre}
\end{figure}

Since $\Ker( \Mapb_n^{\spar_{n}})$ has roughly $K_n \sim 3\spar_n$ edges  by Theorem~\ref{thm:defauts_ker} and the collection of bipointed trees is roughly uniformly distributed amongst those with $n$ edges in total, we can expect that the bipointed trees  have of order $n/\spar_n$ edges each and their diameter is typically of order $\sqrt{n/\spar_n}$.  Although there  are other  interesting  scales to look at, going from the  microscopic scale, or local convergence, to the diameter scale, see precisely Proposition~\ref{prop:limites_arbres} and Conjecture~\ref{conjecture}, we focus here on the geometric structure of $ \Mapb_n^{\spar_n}$ at the mesoscopic scale $ \sqrt{n/\spar_n}$. 

Let us construct the limits that appear in the next theorem; we refer to Section~\ref{sec:mesoscopic} for more details and explanations and to~\ref{fig:mesoscopic} for an illustration.
Similar constructions have been encountered for scaling limits of mean-field random graphs at criticality~\cite{ABBG10} or at the discrete level in high-genus unicellular maps~\cite{Lou21}.
Start either from the dual of the UIPT (type 1) in the planar case, or from the three-regular tree  in the unicellular case; this will play the role of the kernel. Then in order to take into account the root of the full map, let us modify their root edge by inserting a middle vertex and attaching a dandling leaf to it on its right to get a infinite (and random in the planar case) map denoted by $\mathbb{T}_{ \mathrm{Plan}}$ or $\mathbb{T}_{ \mathrm{Unic}}$ depending on the model. We then replace each edge by an independent copy of a Brownian Continuum Random Tree (CRT) with two marked points, whose volume is exponentially biased, i.e.~we glue these trees by their distinguished points according to the graph structure. Note that in this construction, the edges of the graph become independent real segments, and their length is simply distributed according to an exponential law of mean $1/\sqrt{6}$ (with Brownian CRT's attached all along). The resulting locally compact metric space is denoted by $ \mathfrak{F}_{ \mathrm{Plan}}$ in the planar case and by $ \mathfrak{F}_{ \mathrm{Unic}}$ in the unicellular case, and is pointed at the extremity of the CRT grafted on the dangling leaf.

\begin{figure}[!ht]\centering
\includegraphics[width=.9\linewidth]{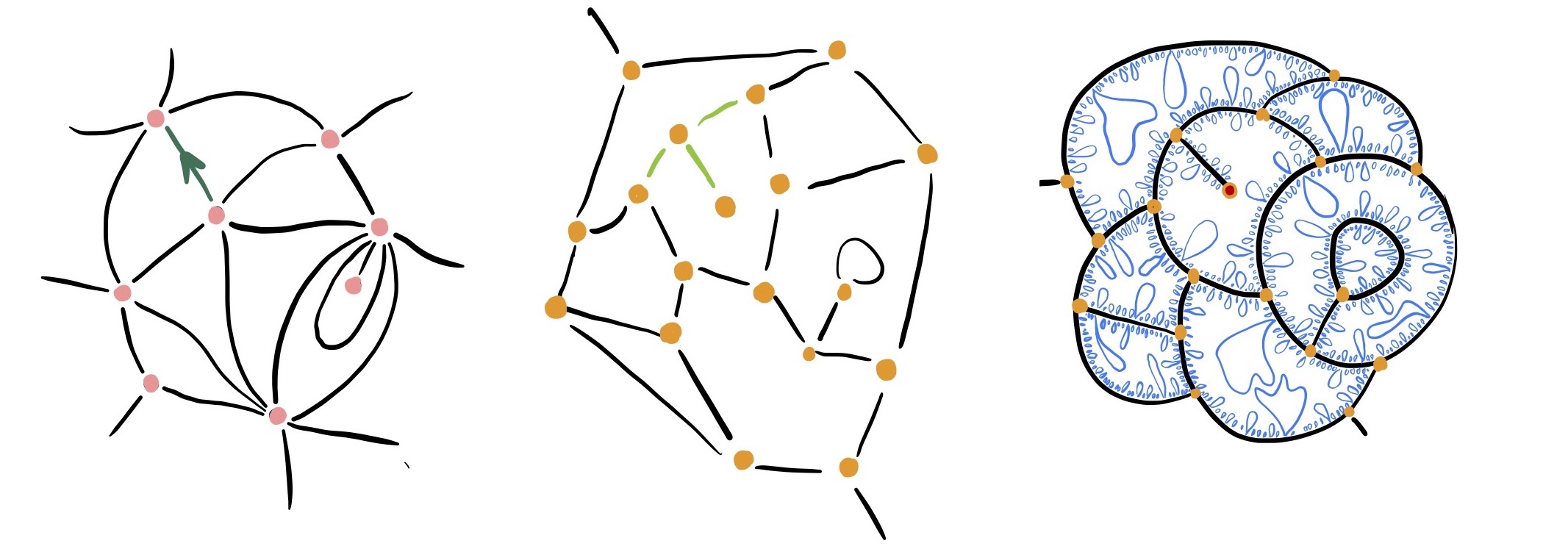}
\caption{
Left: a piece of the UIPT. Middle: the corresponding piece of its dual after performing the root transformation (in green) to get $\mathbb{T}_{ \mathrm{Plan}} $. Right: the pointed metric space $ \mathfrak{F}_{ \mathrm{Plan}}$ build out from $ \mathbb{T}_{ \mathrm{Plan}}$ by replacing edges with random trees.}
\label{fig:mesoscopic}
\end{figure}

\begin{theorem}[Mesoscopic scaling limit]
\label{thm:intermediate}
Under~\eqref{eq:standing} the convergences in distribution
\[\sqrt{ \frac{\spar_{n}}{n}} \cdot  \Map_{n}( \spar_{n},0) \cvloi \mathfrak{F}_{ \mathrm{Plan}}
\qquad \text{and} \qquad 
\sqrt{ \frac{\spar_{n}}{n}} \cdot  \Map_{n}\left(1,\frac{\spar_{n}-1}{2} \right) \cvloi \mathfrak{F}_{ \mathrm{Unic}}\]
hold in the local pointed Gromov--Hausdorff topology.
\end{theorem}

Let us refer the reader to e.g.~\cite{BBI01,CLG14,BMR19} for details on the local pointed Gromov--Hausdorff. Thus, roughly speaking, we can say that the geometry of $ \Mapb_n^{\spar_n}$ at the scale $ \sqrt{n/\spar_n}$ is a mixture of two features: a discrete part, coming from  the local limit of the kernel, and a continuous part coming from the faces which collapse on trees due to the sparse nature of the maps.

It is likely that a similar result holds in the broader context of random maps $\Map_n( \face_n, \gen_n)$ as soon as $ 1 \ll \face_n+2 \gen_n \ll n$. When $  \face_n \ll \gen_n$ or $ \gen_n \ll \face_n$ we believe that the same limits as above should be observed. However, if $\face_n/\gen_n \rightarrow  \alpha>0$, then the kernel of such maps should be essentially trivalent maps (i.e.~whose defect number is negligible compared to the size) with genus proportional to the number of faces. We conjecture that the local limit of such maps is given by the Planar Stochastic Infinite Triangulation (PSHIT) of~\cite{Cur16} with the appropriate parameter, see the work~\cite{BL21} for the purely trivalent case. The mesoscopic scaling limits of $ \Map_n( \face_n, \gen_n)$ in this last regime should then be obtained by replacing edges of the dual of these PSHIT's by bipointed CRT's as above.
This would require to extend the technical estimates of Section~\ref{sec:defauts_bis}. However, besides this, our proofs are robust enough to handle these regimes and several quantities are universal, such as the law of the bipointed Brownian CRT's.

\paragraph{Acknowledgments.}
We thank Thomas Budzinski for sharing early stages of his work~\cite{Bud21}, \'Eric Fusy for the reference~\cite{BCR93} as well as Charles Bordenave and Bram Petri for the pointer to~\cite[Theorem 2.19]{Wor99}. 
Finally, we thank the CIRM for its hospitality in January 2021 when this work was first triggered. 
The first author is supported by ERC 740943 GeoBrown.

\section{Core--Kernel decomposition and enumeration lemmas}
\label{sec:enumeration}

In this section, we describe the exact laws of the core and kernel of a uniform random map $ \Map_n(\face,\gen)$ with $n$ edges, $\face$ faces, and genus $\gen$. We then prove some technical estimates on the number of such maps that share a given kernel, which will be used later to prove our main theorems.
Let us stress that the results in this section are valid for 
all values of $\face \ge 1$ and $\gen \ge 0$.

\subsection{Law of the kernel}
\label{sec:decomposition}

Let us explain in more details the core--kernel decomposition of a map, see also~\cite{CMS09,Cha10}. Fix a (rooted) map $\kernel$ with $\face$ faces and genus $\gen$, and whose vertices all have degree at least $3$. We let $k$ denote its number of edges. Then all maps $\map$  such that $ \Ker( \map) = \kernel$ are uniquely constructed as follows. We first fix $c \geq k$, which will correspond to the number of edges of the core.

To construct $\Core(\map)$, the core of the map, from the kernel, each edge $e$ of $\kernel$ is split into say $n_{e}\geq 1$ consecutive edges by inserting vertices of degree $2$, with  $\sum_{e \in \mathrm{Edges}( \kernel)} n_{e} = c$. To count the number of possible cores, since the edges of $ \kernel$ can unambiguously be indexed by $\{1,2, \dots , k\}$,   then the number of ways of performing this splitting step is equal to the number of ways one can split a discrete cycle of $c$ edges into $k$ parts of length at least $1$, see Figure~\ref{fig:split_aretes}. Also, since we are working with rooted maps, then in order to recover the root edge of the core, we must further distinguish one of the edges produced when splitting the root edge of $\kernel$ (and we keep the same orientation). 
Consequently, the number of ways to get a given core with $c$ edges from the kernel $\kernel$ with $k$ edges equals
\begin{equation}\label{eq:coeffbinom}
\sum_{i=1}^{c-k+1} i \binom{c-i-1}{k-2} = \binom{c}{k}.
\end{equation}
Another way of describing this is as follows. First add a new vertex in the middle of the root edge of the kernel and declare the new root edge to be the one pointing towards this new vertex, with the same origin as the previous root edge. Then index the edges of this modified kernel and represent it as a segment (instead of a cycle). Expand now the $k+1$ edges of this segment into chains with $c+1$ edges in total. This amounts conversely to splitting a segment with $c+1$ edges by choosing $k$ vertices amongst the $c$ inner vertices in total. The expanded kernel is then the core, with the same modification at the root as the kernel.

\begin{figure}[!ht]\centering
\includegraphics[width=.6\linewidth]{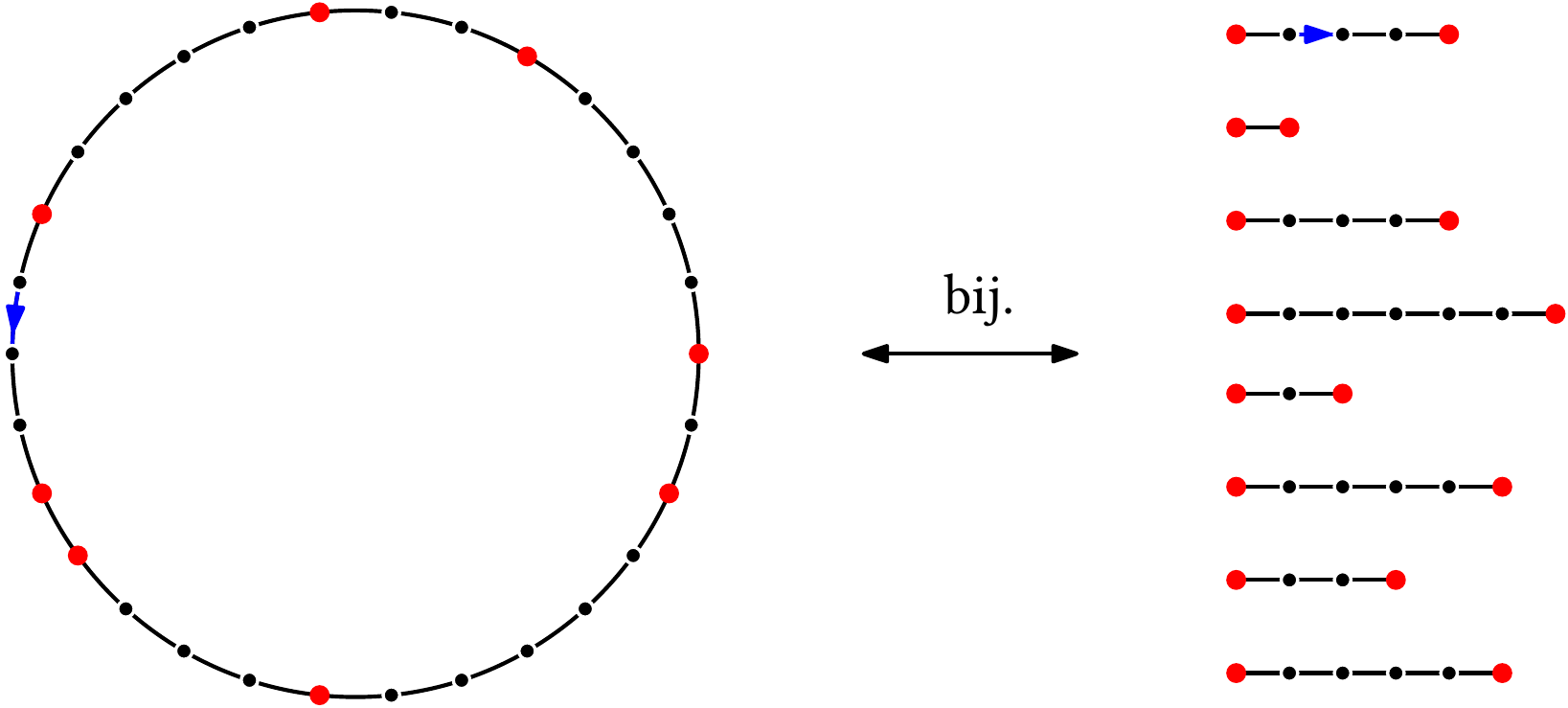}
\caption{A combinatorial representation of the repartition of the edges of the core in the kernel explaining the relation~\eqref{eq:coeffbinom}: Choosing $k$ points amongst $c$ possibilities splits a discrete cycle of $c$ edges into $k$ ordered components, the first one being that containing the root edge (in blue above).}
\label{fig:split_aretes}
\end{figure}

Once the core, with $c$ edges, is constructed, in order to recover the entire map $\map$, it remains to graft a plane rooted tree on each one of its $2c$ corners. Specifically, given an enumeration of the corners of the core from $1$ to $2c$, with $1$ being canonically the corner to the right of the tip of the root edge, we graft on the $i$'th corner a rooted plane tree $T_i$, with say $t_i \ge 0$ edges. Their size must satisfy $t_{1} + \dots + t_{2c} = n-c$. As above, we also need to keep track of the root edge of the map; we either keep the root edge of the core, or we choose one oriented edge in the tree $T_{1}$. This is equivalent to distinguishing a number in $\{0, 1, \dots , 2 t_{1}\}$.

It is classical to encode an ordered forest by a Dyck path, that follows the contour of each tree successively, with an extra negative step after each tree. Our forest is thus encoded by a path with increments either $+1$ or $-1$, which ends by hitting $-2c$ for the first time at time $2n$, with a distinguished time $k \in \{0,1, \dots , \tau-1\}$ where $\tau$ is the hitting time of $-1$ (in order to take into account the rooting). By the classical cycle lemma, this is equivalent to taking a $\pm1$ path starting at $0$ and ending at $-2c$ at time $2n$, and cyclicly shift it at the first time it reaches its overall minimum, see e.g.~\cite[Chapter~6]{Pit06} for details and Fig.~\ref{fig:contour} for an illustration.
Hence, the number of maps with $n$ edges that share a given core with $c$ edges is equal to the number of the latter paths, which is simply
\[\binom{2n}{n+c}.\]

\begin{figure}[!ht]\centering
\includegraphics[width=.9\linewidth]{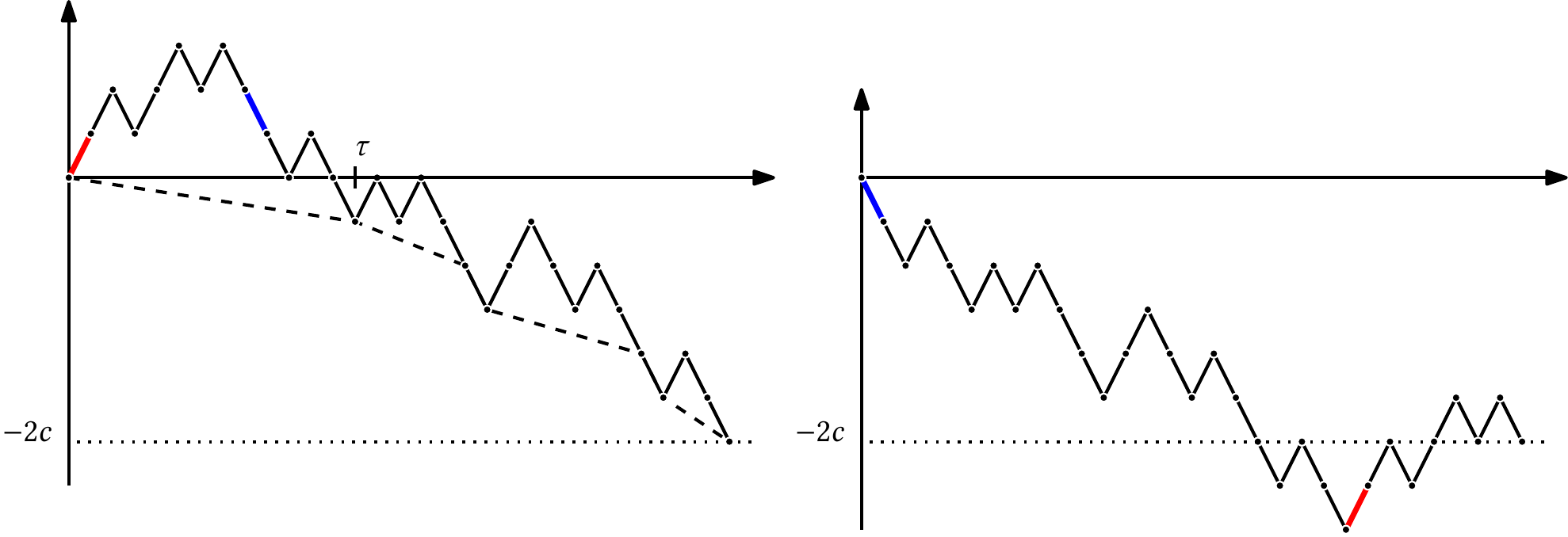}
\caption{A $\pm1$ path starting at $0$ and ending at $-2c$ at time $2n$ (right) is equivalent, after a Vervaat re-rooting at the first hitting time of the minimum to (left) a first passage bridge starting at $0$ and hitting $-2c$ for the first time at $2n$ together with a distinguished time $k \in \{0,1, \dots , \tau-1\}$ where $\tau$ is the hitting time of $-1$. By a classical bijection, those objects are in bijection with ordered plane forests with $n$ edges and $2c$ trees, together with either an oriented edge in the first tree, or another mark.}
\label{fig:contour}
\end{figure}

Observe that a map, its kernel and its core all have the same number $\face$ of faces and the same genus $\gen$, hence the same sparsity parameter $\spar = \face+2\gen$. Let us reformulate the core--kernel decomposition in probabilistic terms. For integers $1\leq k  \leq c \leq n$ let us set 
\begin{equation}\label{eq:phins}
\varphi_{n}(c,k) \coloneqq \binom{c}{k} \binom{2n}{n+c}
\qquad\text{and}\qquad
\Phi_n(k) \coloneqq  \sum_{c \geq k} \varphi_{n}(c,k).
\end{equation}
Let us extend them both by $0$ to all values of $c$ and $k$.
Then $\varphi_n(c,k)$ denotes the number of maps with $n$ edges, with a core with $c$ edges, which have a given kernel $\kernel$ with $k$ edges. Further, $\Phi_n(k)$ is the number of maps with $n$ edges which share such a given kernel. 

Recall from the Introduction that 
 $\mathcal{T}_{\defect}(\face,\gen)$ stands for the set of all rooted maps with $ \face$ faces, in genus $ \gen$, whose vertices all have  degree at least $3$,  and $\defect$ defects, i.e.~such that $\sum_{v \in \mathrm{Vertices}} (\deg(v)-3) = \defect$, or equivalently which have $3(\face+2\gen) - \defect - 6$ edges. 
We shall let $\mathcal{T}(\face,\gen) = \bigcup_{\defect \ge 0} \mathcal{T}_{\defect}(\face,\gen)$. 
We infer that the number $\#\mathfrak{M}_n(\face,\gen)$ of maps with $n$ edges, $\face$ faces, and genus $\gen$ equals
\[\#\mathfrak{M}_n(\face,\gen) = \sum_{d \ge 0} \# \mathcal{T}_{\defect}(\face,\gen)\, \Phi_n(3(\face+2\gen) - \defect - 6).\]
By Euler's formula, lower bounding the number of vertices by $1$, any map in $\mathcal{T}(\face,\gen)$ has at least $\face+2\gen-1$ edges so $\mathcal{T}_{\defect}(\face,\gen)$ is empty as soon as $\defect > 2(\face+2\gen)-5$.
Recall that $\Map_{n}(\face,\gen)$ denotes a map in $\mathfrak{M}_n(\face,\gen)$ sampled uniformly at random. 
The above discussion can be reformulated as follows.

\begin{proposition}\label{prop:loi_core_kernel}
Fix $n\ge1$, $\face \geq 1$, $\gen\ge0$, and $\kernel \in \mathcal{T}(\face,\gen)$; let $k$ denote the number of edges of $\kernel$.
\begin{enumerate}
\item We have
\[ \P \left( \Ker( \Map_n(\face,\gen)) = \kernel \right) = \frac{\Phi_{n}(k)}{\# \mathfrak{M}_n(\face,\gen)}.\]
Consequently, conditionally given its number of edges, say $\#\mathrm{Edges}(\Ker( \Map_n(\face,\gen))) = k$, the kernel is uniformly distributed over $ \mathcal{T}_\defect(\face,\gen)$ with $\defect=3(\face+2\gen)-6-k$. 

\item For any $c \geq k$, we have
\[\P \left(\#\mathrm{Edges}(\Core( \Map_n(\face,\gen))) = c  \;\middle|\; \Ker( \Map_n(\face,\gen)) = \kernel \right) = \frac{\varphi_{n}(c,k)}{\Phi_{n}(k)}.\]
Furthermore, conditionally given $\Ker( \Map_n(\face,\gen)) = \kernel$ and $\#\mathrm{Edges}(\Core( \Map_n(\face,\gen))) = c$, the core is obtained by replacing each edge of $\kernel$ by a chain of edges, 
whose lengths are given by $N_0+N_1-1$ for the root edge and $N_2, \dots, N_k$ for the other edges, where $(N_0, \dots, N_k)$ has the uniform distribution on the set of positive integer vectors which sum up to $c+1$.

\item Finally, conditionally given the core, with say $c$ edges, let us sample uniformly at random a forest $(T^n_{1}, \dots, T^n_{2c})$ with $n-c$ edges together with $\vec{e}$ being either an oriented edge in the tree $T^n_1$ or the mark $\ast$.
Then attach the above trees in the corners of the core, with $T^n_1$ in the corner to the right of the tip of the root edge, and root this map at $\vec{e}$ if it is different from $\ast$, and at the root edge of the core otherwise. Then this map has the law of $\Map_n(\face,\gen)$.
\end{enumerate}
\end{proposition}

\subsection{Technical estimates}

In this section we gather a few estimates about $\Phi_n(k) = \sum_{c \geq k} \varphi_{n}(c,k)$ and locate the values of $c$ which form the main contribution in this sum.  In particular we will prove that when $1 \ll k_{n} \ll n$, we have 
\begin{equation}\label{eq:ratioPhi1}
\frac{ \Phi_{n}(k_{n}+1)}{ \Phi_{n}(k_{n})} \equi \sqrt{ \frac{n}{2 k_{n}}}.
\end{equation}
Combined with the estimates on the number of near-trivalent maps in the next section, this will provide the basis of the proof of Theorem~\ref{thm:defauts_ker}. All these estimates are based on the exact formulas~\eqref{eq:phins} and rather elementary (but tedious) manipulations using Stirling's formula.

Fix $1 \leq k \leq n$.  We start by computing  the ratios of consecutive summands $\varphi_n(C,k)$ involved in the definition of $\Phi_n(k)$:
\begin{equation}\label{eq:def_Rat}
\mathrm{Rat}_{n}(C,k) \coloneqq  \frac{\varphi_n(C,k)}{\varphi_n(C-1,k)} = \frac{C (1 - C + n)}{(C + n) (C - k)},
\end{equation}which we extend to the whole interval $[k+1,n]$.
Then for every $C \in [k+1,n]$ we have 
\begin{equation}\label{eq:derivlog}
\frac{\partial}{\partial C} \log  \mathrm{Rat}_n(C,k)
= - \frac{k}{C(C-k)}- \frac{2n+1}{(n-C+1)(C+n)}
< 0.
\end{equation}
Therefore the function $C \mapsto  \mathrm{Rat}_n(C,k)$ is decreasing on $[k,n]$. One can check that it crosses level $1$ in the interval $[C_{n,k},C_{n,k}+1)$, where $C_{n,k}$ is defined as 
\begin{equation}\label{eq:argmax_phi}
C_{n,k} = \left\lfloor \frac{1}{4} \left( 1+k + \sqrt{ (k+1)^{2} + 8 n k } \right) \right\rfloor.
\end{equation}
Hence the maximal value of $\varphi_n(C,k)$ is attained at $C=C_{n,k}$ and it equals
\[\max_{C \ge k} \varphi_n(C,k) = \varphi_n(C_{n,k},k) = \binom{C_{n,k}}{k} \binom{2n}{n+C_{n,k}}.\]

\begin{lemma}\label{lem:phi_et_Phi}
When $k_n\to\infty$ with $k_n/n\to0$, we have
\begin{equation}\label{eq:equiv_kn}
\sqrt{\frac{k_n}{n}} \cdot \frac{\varphi_n(C_{n,k_n+1},k_n+1)}{\varphi_n(C_{n,k_n},k_n)} \cv \frac{1}{\sqrt{2}}
\qquad\text{and}\qquad
\frac{1}{\sqrt{n}} \cdot \frac{\Phi_n(k_n)}{\varphi_{n}(C_{n,k_n},k_n)} \cv \sqrt{\frac{\pi}{2}}
.\end{equation}
Also, there exist two universal constants, say $0 < a < A < \infty$, such that for every integers $n$ and $k \le n/32$, it holds
\begin{equation}\label{eq:bornes_ratio}
a\,\sqrt{\frac{n}{k}} < \frac{\varphi_n(C_{n,k+1},k+1)}{\varphi_n(C_{n,k},k)} < A\,\sqrt{\frac{n}{k}}
\qquad\text{and}\qquad
a\sqrt{n} < \frac{\Phi_n(k)}{\varphi_{n}(C_{n,k},k)} < A \sqrt{n}.
\end{equation}
Consequently, in these respective regimes,
\[\frac{a^2}{A}\,\sqrt{\frac{n}{k}}
< \frac{\Phi_n(k+1)}{\Phi_n(k)} 
< \frac{A^2}{a}\,\sqrt{\frac{n}{k}}
\qquad\text{and}\qquad
\sqrt{\frac{k_n}{n}} \cdot \frac{\Phi_n(k_n+1)}{\Phi_n(k_n)} \cv \frac{1}{\sqrt{2}}
.\]
\end{lemma}

\begin{proof}
Assume that $k_n\to\infty$ with $k_n/n\to0$. By using the explicit expression~\eqref{eq:argmax_phi},  we get
\begin{equation}\label{eq:estimee_Ckn}
C_{n,k_n} \equi \sqrt{\frac{n k_n}{2}}
\qquad\text{and}\qquad
C_{n,k_n+1} - C_{n,k_n} \equi \sqrt{\frac{n}{8 k_n }}
\end{equation}
Let us first compare $\varphi_{n}(C_{n,k_{n}+1},k_{n}+1)$ with $\varphi_{n}(C_{n,k_{n}},k_{n})$.  By definition, we have
\begin{equation}\label{eq:ratio_phi}
\frac{\varphi_{n}(C_{n,k+1},k)}{\varphi_{n}(C_{n,k},k)}
= \frac{\binom{C_{n,k+1}}{k}}{\binom{C_{n,k}}{k}} \frac{\binom{2n}{n+C_{n,k+1}}}{\binom{2n}{n+C_{n,k}}}
= \prod_{j=1}^{C_{n,k+1}-C_{n,k}}  \left( 1+\frac{k}{C_{n,k}+j-k} \right) \left( 1- \frac{C_{n,k+1}+C_{n,k}}{C_{n,k}+j+n}\right),
\end{equation}
By taking $k=k_{n}$ and simply bounding each term by taking either $j=1$ or $j=C_{n,k_{n}+1}-C_{n,k_{n}}$ and using the estimates~\eqref{eq:estimee_Ckn} , we get
\begin{align*}
\ln \frac{\varphi_{n}(C_{n,k_n+1},k_n)}{\varphi_{n}(C_{n,k_n},k_n)}
&\equi \sqrt{\frac{n}{8k_n}} \ln\left(\left(1+\sqrt{\frac{2k_n}{n}}\right) \left(1-\sqrt{\frac{2k_n}{n}}\right)\right)
\equi - \sqrt{\frac{k_n}{2n}}
,\end{align*}
which converges to $0$. In addition,
\[\frac{\varphi_{n}(C_{n,k_n+1},k_n+1)}{\varphi_{n}(C_{n,k_n+1},k_n)}
= \frac{\binom{C_{n,k_n+1}}{k_n+1}}{\binom{C_{n,k_n+1}}{k_n}}
= \frac{C_{n,k_n+1} - k_n}{k_n+1}
\equi \sqrt{\frac{n}{2k_n}}
.\]
Therefore
\[\frac{\varphi_{n}(C_{n,k_n+1},k_n+1)}{\varphi_{n}(C_{n,k_n},k_n)}
= \frac{\varphi_{n}(C_{n,k_n+1},k_n+1)}{\varphi_{n}(C_{n,k_n+1},k_n)} \frac{\varphi_{n}(C_{n,k_n+1},k_n)}{\varphi_{n}(C_{n,k_n},k_n)} 
\equi \sqrt{\frac{n}{2k_n}}
,\]
which establishes the first convergence in~\eqref{eq:equiv_kn}.

In the rest of the proof we assume that $n$ and $k$ are integers with $k \le n/32$.  First note that it holds
\[\sqrt{\frac{n k}{8}}
\le C_{n,k} 
\le 2 \sqrt{n k}
\qquad\text{and}\qquad
\frac{1}{6} \sqrt{\frac{n}{k}} \le C_{n,k+1} - C_{n,k} \le 2 \sqrt{\frac{n}{k}}.\]
For the last two bounds, one can use that $x/3 \le \sqrt{1+x} - 1 \le x/2$ for every $x \in[0,1]$.

The  bounds on the left-hand side of~\eqref{eq:bornes_ratio} are obtained in a similar way as before: by bounding each term in the product~\eqref{eq:ratio_phi}  by taking either $j=1$ or $j=C_{n,k+1}-C_{n,k}$, using the bounds on $C_{n,k}$ and $C_{n,k+1}-C_{n,k}$, and then using that $\exp(-x/(1-x)) \le 1-x \le \exp(-x)$ for every $x \in [0,1]$, it is straightforward, yet tedious, to show that the ratios
$\varphi_{n}(C_{n,k+1},k)/\varphi_{n}(C_{n,k},k)$
are bounded away from $0$ and infinity by universal constants.

To show the second convergence of~\eqref{eq:equiv_kn} and the bounds on the right-hand side of~\eqref{eq:bornes_ratio}, we start by comparing $\Phi_n(k)$ with $\varphi_{n}(C_{n,k},k)$.
Recall the quantity $\mathrm{Rat}_{n}(C,k)$ introduced in~\eqref{eq:def_Rat}. 
Using the bounds on $C_{n,k}$ above, one can crudely upper bound the derivative in~\eqref{eq:derivlog} around $C_{n,k}$ as follows: For every integers $n \ge k \ge 1$ and for every real number $x$ such that $|x| \le \sqrt{nk}$, we have 
\[\frac{\partial}{\partial C} \log \mathrm{Rat}_n(C_{n,k}+x,k)
\le - \frac{k}{(3\sqrt{nk})^2}- \frac{2n+1}{(n+3\sqrt{nk})^2}
\le - \frac{1}{5n}
.\]
Moreover, if $k_n \to\infty$ with $k_n = o(n)$, then 
one can check that, uniformly for $x$ such that $|x| \leq \sqrt{n} k_{n}^{1/4}$, we have
\[n \cdot \frac{\partial}{\partial C} \log  \mathrm{Rat}_n(C_{n,k_{n}}+x,k_{n}) \cv -4.\]
Writing $ \log \mathrm{Rat}_n(C_{n,k}+x,k) = \int_{1}^{x} \frac{\partial}{\partial C} \log \mathrm{Rat}_n(C_{n,k}+u,k) \d u+\log \mathrm{Rat}_n(C_{n,k}+1,k)$ and using the fact that $\log \mathrm{Rat}_n(C_{n,k}+1,k) \leq 0$ by definition of $C_{n,k}$, 
this entails that for any $n \ge k \ge 1$ and any integer $j$ such that $|j| \le \sqrt{nk}$,
\begin{equation} \label{eq:borne_unif_phi}
\varphi_{n}(C_{n,k}+j,k) \le \varphi_{n}(C_{n,k},k) \cdot \exp\left(- \frac{j(j-1)}{10n}\right),
\end{equation}
and uniformly for all integers $j$ such that $|j| \leq \sqrt{n} k_{n}^{1/4}$, we have
\begin{equation}\label{eq:locallimitCLT}
\varphi_{n}(C_{n,k_{n}}+j,k_{n}) \equi \varphi_{n}(C_{n,k_{n}},k_{n}) \cdot \exp\left(-\frac{2j(j-1)}{n}\right).
\end{equation}
Recalling that $k_{n} \to \infty$, we infer from~\eqref{eq:locallimitCLT}:
\[\sum_{|j| \leq \sqrt{n} k_{n}^{1/4}}  \frac{\varphi_{n}(C_{n,k_{n}}+j,k_{n})}{\varphi_{n}(C_{n,k_{n}},k_{n})}
\equi  \sum_{|j| \leq \sqrt{n} k_{n}^{1/4}} \exp\left(- \frac{2j^2}{n}\right)
\equi \sqrt{\frac{n \pi}{2}}
.\]
Also for any values $n \ge k \ge 1$ it holds
\[\sum_{|j| \leq \sqrt{n} k^{1/4}}  \frac{\varphi_{n}(C_{n,k}+j,k)}{\varphi_{n}(C_{n,k},k)}
\le \sqrt{20 n \pi}
.\]
Recall next that $ C \mapsto  \mathrm{Rat}_n(C,k)$ is decreasing and its value at $C_{n,k}+j$ is bounded above by $\exp(-j/(5n))$ as soon as $|j| \le \sqrt{nk}$. Therefore, for any $|j| \ge\sqrt{n}$, by applying~\eqref{eq:borne_unif_phi} at $\pm\sqrt{n}$ we infer that
\[\varphi_{n}(C_{n,k}+j,k) \le \varphi_{n}(C_{n,k},k) \cdot \exp\left(- \frac{|j|-\sqrt{n}}{5\sqrt{n}}\right)
.\]
Consequently, for any $n\ge k \ge 1$,
\[ \sum_{|j| > \sqrt{n} k^{1/4}}  \frac{\varphi_{n}(C_{n,k}+j,k)}{\varphi_{n}(C_{n,k},k)}
\leq \e^{1/5} \sum_{|j| > \sqrt{n} k^{1/4}} \exp\left(- \frac{|j|}{5\sqrt{n}}\right)
\leq 10 \e^{1/5} \sqrt{n} \exp(- k^{1/4}/5)
.\]
In particular,  if $k_n \to\infty$ with $k_n = o(n)$, then we conclude that
\[\frac{\Phi_n(k_n)}{\sqrt{n} \varphi_{n}(C_{n,k_n},k_n)} \cv \sqrt{\frac{\pi}{2}}.\]
In addition, for every $n \ge k \ge 1$ it holds
\[\frac{\Phi_n(k)}{\sqrt{n} \varphi_{n}(C_{n,k},k)}
\le \sum_{|j| < \sqrt{n}}  \frac{\varphi_{n}(C_{n,k}+j,k)}{\sqrt{n} \varphi_{n}(C_{n,k},k)} + \sum_{|j| \ge \sqrt{n}}  \frac{\varphi_{n}(C_{n,k}+j,k)}{\sqrt{n} \varphi_{n}(C_{n,k},k)}
\le 13,\]
since the second sum is bounded by $10$ as above and the first one by $3$ since $\varphi_{n}(C_{n,k},k) = \max_C \varphi_{n}(C,k)$.

On the other hand we can also lower bound the left hand side of the last display when $k \le n/32$. Indeed in this case it holds $k \le C_{n,k}/4$ and $C_{n,k} \le n/4$ so if $|x| \le \sqrt{nk}/16 \le C_{n,k}/4$, then
\[\frac{\partial}{\partial C} \log  \mathrm{Rat}_n(C_{n,k}+x,k)
\ge - \frac{8k}{3 C_{n,k}^2} - \frac{2n+1}{(1+11n/16) n}
\ge - \frac{64}{3 n} - \frac{48}{11n}
\ge -26n.\]
Then as above we infer that
\[\frac{\Phi_n(k)}{\sqrt{n} \varphi_{n}(C_{n,k},k)}
\ge \sum_{|j| \leq \sqrt{nk}/16} \frac{\varphi_{n}(C_{n,k}+j,k)}{\sqrt{n} \varphi_{n}(C_{n,k},k)}
\ge \sum_{|j| \leq \sqrt{nk}/16} n^{-1/2} \exp\left(-\frac{13 j^2}{n}\right)
\ge \int_{|u| \le \sqrt{k}/16} \exp\left(-13 u^2\right) \d u
,\]
and the right hand side is larger than $1/20$.
\end{proof}

\subsection{Two applications}

We gather here two applications of the technical estimates of Lemma~\ref{lem:phi_et_Phi}. The first one gives an asymptotic on the  number $\Phi_{n}(k_{n})$ of maps with $n$ edges which have a given kernel with $k_{n}$ edges, in the regime $k_n = O(n^{1/3})$. This will be used in Section~\ref{ssec:asymp_enum} to establish in turn the asymptotic estimates of  Corollary~\ref{cor:enumeration}.

\begin{corollary}
\label{cor:equivPhikn}
Assume that $k_{n}^3/n \to \gamma \in [0,\infty)$. Then
\[ \Phi_{n}(k_{n})  \enskip \mathop{\sim}_{n \rightarrow \infty} \enskip   
\frac{\e^{- \sqrt{\gamma/2}}}{2\sqrt{2 \pi k_{n}}} \cdot 4^{n}  \cdot \left(  \frac{\e n}{2 k_{n}}\right)^{k_{n}/2}
.\]
\end{corollary}

\begin{proof}
Recall from~\eqref{eq:argmax_phi} the explicit expression of $C_{n,k_{n}}$. Note that $k_{n}^{2}/C_{n,k_{n}} \rightarrow \sqrt{2 \gamma}$ and that $C_{n,k_n}^{3}/n^{2} \rightarrow \sqrt{\gamma/8}$. By Lemma~\ref{lem:phi_et_Phi},
\[ \Phi_{n}(k_{n})  \enskip \mathop{\sim}_{n \rightarrow \infty} \enskip \sqrt{ \frac{\pi n}{2}} \cdot \varphi_{n}(C_{n,k_n},k_n)=    \sqrt{ \frac{\pi n}{2}} \cdot \binom{C_{n,k_{n}}}{k_{n}} \binom{2n}{n+C_{n,k_{n}}}. \]
Using Stirling's formula, we get 
\[ \binom{C_{n,k_{n}}}{k_{n}}  = \frac{1}{\sqrt{2 \pi k_{n}}} \exp \left( -C_{n,k_n} \ln \left( 1- \frac{k_{n}}{C_{n,k_n}}\right)+ k_{n} \ln C_{n,k_n}+ k_{n} \ln \left( 1- \frac{k_{n}}{C_{n,k_n}}\right)-k_{n} \ln k_{n} +o(1) \right).\]
A second order asymptotic expansion of $\ln(1-k_{n}/C_{n,k_n})$ combined with the fact that $k_{n }\ln(C_{n,k_{n}})= \frac{k_{n}}{2} \ln(n k_{n}/2)+\sqrt{\gamma/8}+o(1)$ yields
\[ \binom{C_{n,k_{n}}}{k_{n}}  \enskip \mathop{\sim}_{n \rightarrow \infty} \enskip  \frac{1}{\sqrt{2 \pi k_{n}}} \e^{k_{n}-\sqrt{\gamma/8}}  \left(  \frac{n}{2 k_{n}}\right)^{k_{n}/2}.\]
Similarly, by Stirling's formula,
\begin{align*}
 \binom{2n}{n+C_{n,k_{n}}}  & =  \frac{2^{2n}}{\sqrt{ 2 \pi n}}  \exp \left(-n \ln \left( 1- \frac{C_{n,k_n}^{2}}{n^{2}} \right)-C_{n,k_n}\ln \left( 1+ \frac{C_{n,k_n}}{n} \right) + C_{n,k_n}\ln \left( 1- \frac{C_{n,k_n}}{n}\right)+o(1)\right) \\
 &= \frac{4^n}{\sqrt{ 2 \pi n}}   \exp \left( - \frac{C_{n,k_n}^{2}}{n}+o(1) \right).
 \end{align*}
Using the fact that $C_{n,k_n}^{2}/n=k_{n}/2+ \sqrt{\gamma/8}+o(1)$, we infer that
\[  \binom{2n}{n+C_{n,k_{n}}}  \quad \mathop{\sim}_{n \rightarrow \infty} \quad \frac{1}{\sqrt{ 2 \pi n}} 4^{n} \e^{-k_{n}/2-\sqrt{\gamma/8}}.\]
This completes the proof.
\end{proof}

The second application, in conjunction with Proposition~\ref{prop:loi_core_kernel}, provides a Central Limit Theorem for the  size of the core conditionally given that of the kernel, as mentioned in the introduction.

\begin{corollary}\label{cor:TCL_core}
Let $(k_n)_n$ be a sequence of integers such that both $k_n \to \infty$ and $n^{-1} k_n\to0$ and let $\mathcal{C}_{n}$ denote a random variable with distribution given by
\[ \P( \mathcal{C}_{n}=c)= \frac{\varphi_{n}(c,k_{n})}{ \Phi_{n}(k_{n})}, \qquad c \geq k_{n}.\]
Then
\[2 \cdot \frac{ \mathcal{C}_{n}-c_n}{ \sqrt{n}} \cvloi N(0,1),
\qquad\text{where}\qquad
c_n = \frac{k_n + \sqrt{k_n^{2} + 8 n k_n }}{4} ,
\]
and where $N(0,1)$ is the standard Gaussian distribution.
\end{corollary}

\begin{proof}
Combining Lemma~\ref{lem:phi_et_Phi} and the estimate~\eqref{eq:locallimitCLT}, we get that for every $x\in\R$, 
\[\P(\mathcal{C}_{n}=\lfloor C_{n,k_{n}}+x\sqrt{n} \rfloor) \equi \sqrt{\frac{2}{n\pi}} \exp\left(-2 x^{2}\right),\]
where $C_{n,k}$ is defined in~\eqref{eq:argmax_phi}. One easily checks that $C_{n,k_{n}} = c_n + o(\sqrt{n})$, and the convergence in distribution immediately follows from the previous local estimate.
\end{proof}

\section{Near-trivalent maps}
\label{sec:defauts_bis}

We gather in this section some results on near-trivalent maps, i.e.~maps with a small defect.  Specifically, echoing~\eqref{eq:ratioPhi1}, we shall control the growth of the ratio
\begin{equation}\label{eq:ratio}
\frac{\# \mathcal{T}_{\defect+1}( \face, \gen)}{ \# \mathcal{T}_{\defect}( \face, \gen)}
\end{equation}
in the planar case $\gen=0$ and $\defect \ll \face \ll n$ as well as in the unicellular case $\face=1$ and $\defect \ll \gen \ll n$.

We shall also prove the local convergence of uniformly distributed near-trivalent maps in the planar case towards the dual of the UIPT and in the unicellular case towards the three regular tree. The main technique we develop in order to control the ratios~\eqref{eq:ratio}  as a function of $ \defect$ is a contraction operation that, starting from a trivalent map, builds a map with defects by contracting edges, which we next introduce.

 Recall that $\mathcal{T}_{\defect}{(\face,\gen)}$ denotes the set of all rooted maps with $ \face$ faces, in genus $ \gen$, whose vertices all have  degree at least $3$,  and $\defect$ defects, i.e.~such that $\sum_{v \in \mathrm{Vertices}} (\deg(v)-3) = \defect$. Also, we denote by $\Triv_{ \defect}(\face,\gen)$ a uniformly chosen element in this set.
Recall the sparsity parameter $\spar=\face+2\gen$; by~\eqref{eq:Euler_defauts} maps in $\mathcal{T}_{\defect}{(\face,\gen)}$ have $3\spar-\defect-6$ edges and $2\spar-\defect-4$ vertices.

\subsection{The contraction operation}
\label{ssec:contrac}
\label{sec:contraction}

Fix $\face \ge 1$, $\gen \ge 0$, and $\defect \ge 1$, and let $\mathfrak{t}_0 \in \mathcal{T}_0(\face, \gen)$ be a trivalent map with $\face$ faces and genus $\gen$.
Let $(e_1, \dots, e_\defect)$ be an ordered list of edges of $\mathfrak{t}_0$ which are all different from one another and all different from the root edge, and which form a forest in $\mathfrak{t}_0$; in particular, they contain no loop nor multiple edges. We henceforth call such a subset $(e_1, \dots, e_\defect)$ of edges a ``good'' subset of edges. Let $\mathrm{Contr}(\mathfrak{t}_0 ; e_1, \dots, e_\defect)$ denote the map obtained from $\mathfrak{t}_0$ by contracting these $\defect$ edges, i.e.~by removing these edges and merging their endpoints. Observe that this map has the same number $\face$ of faces and the same genus $\gen$ as $\mathfrak{t}_0$ (so the same sparsity parameter $\spar$ as well), but it now has defect number $\defect$; it is naturally rooted at the root edge of $\mathfrak{t}_0$, see Figure~\ref{fig:defect-contract} for an example. 

\begin{figure}[!ht]\centering
\includegraphics[width=.8\linewidth]{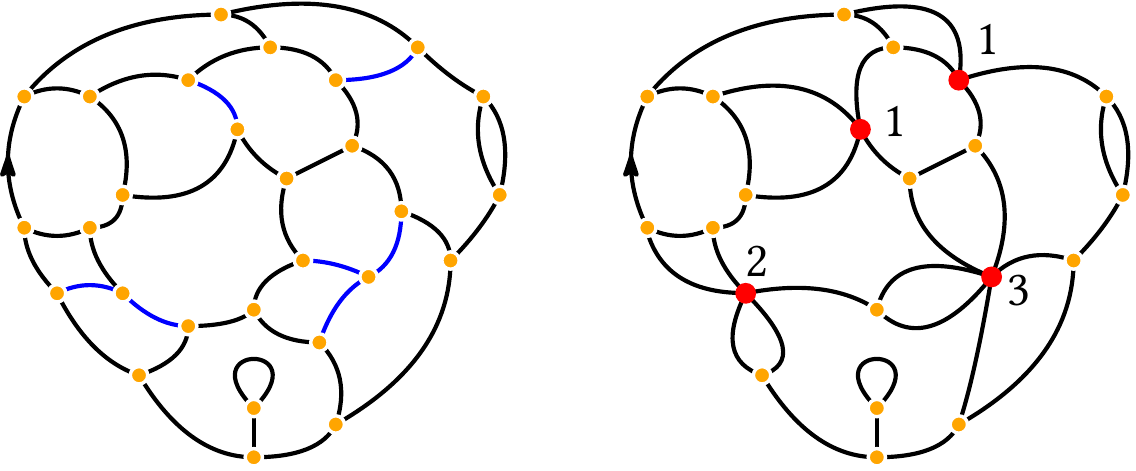}
\caption{Left: A trivalent map with a ``good'' subset of  $\defect=7$ distinguished  edges in blue. Right: the trivalent map with $\defect$ defects obtained by contracting these blue edges; the positive defects are indicated  next to  vertices with degree at least $3$.}
\label{fig:defect-contract}
\end{figure}

\begin{figure}[!ht]\centering
\includegraphics[width=.8\linewidth]{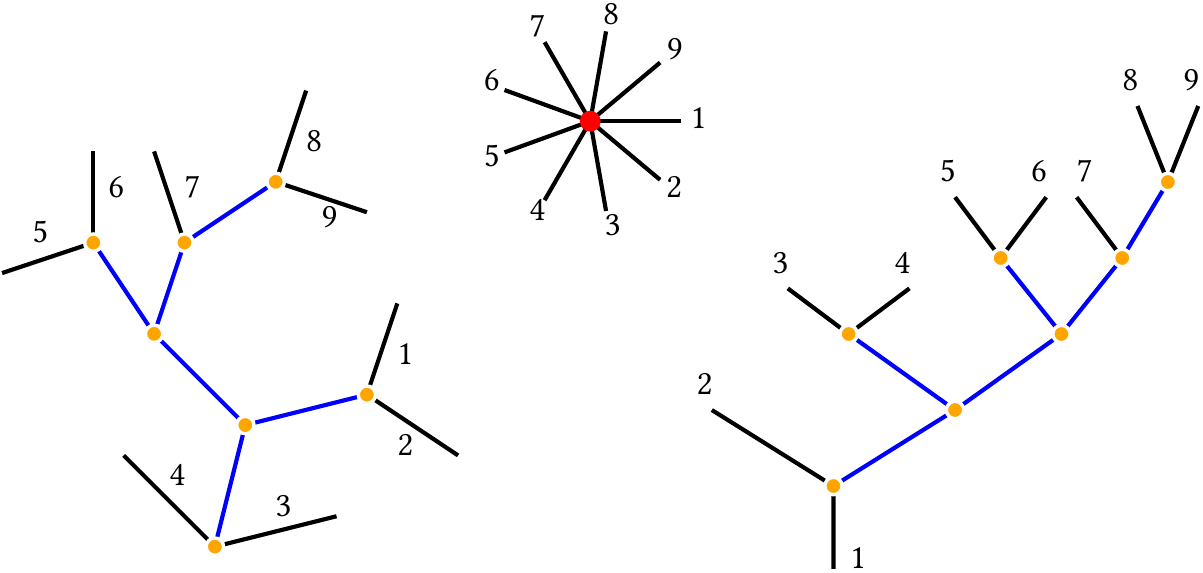}
\caption{A vertex with degree $4$ or larger (here $9$), can be blown up as a trivalent unrooted plane tree with as many leaves (left), which is equivalent to a binary tree after planting it at one leaf.
Conversely, in any tree, one contracts the internal edges to recover the original vertex.}
\label{fig:unzip}
\end{figure}

This mapping is clearly surjective since every map in $\mathcal{T}_\defect(\face, \gen)$ can be obtained from a trivalent map by contracting $\defect$ edges. In general, is not injective since several trivalent maps with a good subset of edges may yield the same map after contraction, as seen on Figure~\ref{fig:unzip}.
More precisely, fix a map $ \mathfrak{t}_\defect \in \mathcal{T}_\defect(\face, \gen)$, consider a vertex $v \in \mathfrak{t}_\defect$ with degree $\ell+3$ with $\ell \geq 1$ and label the edges around $v$ in a canonical order given $\mathfrak{t}_\defect$. Then the number of ways to locally ``blow-up'' the vertex $v$ in a trivalent tree equals the number of rooted plane binary trees (each vertex has either $0$ or $2$ children) with $\ell+2$ leaves, and the number of such trees is given by the Catalan number $\mathrm{Cat}(\ell+1) = \frac{1}{\ell+2} \binom{2\ell+2}{\ell+1}$. The previous considerations show that the number of trivalent maps $\mathfrak{t}_0$ with a distinguished ordered good subset edges $e_1, \dots, e_\defect$ that give rise to the same map $\mathfrak{t}_\defect$ with $\defect$ defects is precisely equal to\[\defect! \prod_{v \in \mathfrak{t}_\defect} \mathrm{Cat}(\deg v - 2).\]
This translates into the following relation: For every nonnegative function $F$ on $\mathcal{T}_\defect(\face, \gen)$, it holds that
\begin{equation}\label{eq:bijection_defauts}
\sum_{\mathfrak{t}_\defect \in \mathcal{T}_\defect(\face, \gen)} F(\mathfrak{t}_\defect)
= \sum_{\mathfrak{t}_0 \in \mathcal{T}_0(\face, \gen)} \sum_{(e_1, \dots, e_\defect) \text{ good}} \frac{F(\mathrm{Contr}(\mathfrak{t}_0 ; e_1, \dots, e_\defect))}{\defect! \prod_{v \in \mathrm{Contr}(\mathfrak{t}_0 ; e_1, \dots, e_\defect)} \mathrm{Cat}(\deg v - 2)}
.\end{equation}

\paragraph{A few crude estimates.}
Let us deduce a few preliminary estimates from the contraction operation before being more precise in the next sections.
The contraction operation can also be performed on a map with defects: Starting from a map with $\defect-1$ defects and one distinguished edge which is neither its root nor a loop, by contracting this edge one obtains a map with $\defect$ defects and two non consecutive distinguished corners around a vertex. Note that this operation is bijective. Hence, by simply bounding above the number of  edges by $3\spar$ on the one hand and bounding below the number of admissible pairs of corners as $\sum_{\textrm{vertex } v} \deg(v) (\deg(v)-3) / 2 \ge 2\defect$ on the other hand, we infer that for any $(\face, \gen)$ and $\defect$,
\begin{equation}\label{eq:crude1}
3 \spar \cdot \# \mathcal{T}_{\defect-1}(\face,\gen) \ge 2 \defect \cdot \# \mathcal{T}_{\defect}(\face,\gen).
\end{equation}

If $ \mathfrak{t}_0$ is a trivalent map, we set 
\[G_\defect(\mathfrak{t}_0) = \sum_{(e_1, \dots, e_{\defect}) \text{ good}} \frac{1}{\defect! \prod_{v \in \mathrm{Contr}(\mathfrak{t}_0 ; e_1, \dots, e_{\defect})} \mathrm{Cat}(\deg v - 2)}.\]
Note that $2^{n-1} \le \mathrm{Cat}(n) \le 4^{n-1}$ and that $\sum_{v \in \mathrm{Contr}(\mathfrak{t}_0 ; e_1, \dots, e_{\defect})} (\deg v - 3) = \defect$ by definition of the defect number. Also, on the one hand, the number of ordered good subsets of $\defect$ edges in a trivalent map is certainly smaller than the number of $\defect$ tuples of edges.
On the other hand, the number of such subsets is greater than the number of ways to distinguish $\defect$ edges which are not loops and such that none of them is incident to another one. Since any loop is adjacent to a non-loop edge, then in the worst case two loops may share a common neighbouring edge so the proportion of non loop edges is at least $1/3$ and each of them is incident to at most $4$ other edges.
We infer the following crude bounds for any trivalent map $\mathfrak{t}_0$:
\[4^{-\defect} \binom{\lfloor \#\mathrm{Edges}(\mathfrak{t}_0)/12\rfloor}{\defect} \le G_\defect(\mathfrak{t}_0) \le 2^{-\defect} \binom{\#\mathrm{Edges}(\mathfrak{t}_0)}{\defect}.
\]
Recall from~\eqref{eq:Euler_defauts} that $\#\mathrm{Edges}(\mathfrak{t}_0) = 3(\spar-2)$ where $\spar = \face+2\gen$ is the sparsity parameter.

Let $\Triv_\defect$ denote a uniform random map in $ \mathcal{T}_{\defect}( \face,\gen)$. Then the identity~\eqref{eq:bijection_defauts} applied twice, once with $F=1$, yields
\[\E[F(\Triv_{ \defect})]
= \frac{1}{\sum_{\mathfrak{t}_0 \in \mathcal{T}_0(\face, \gen)} G_\defect(\mathfrak{t}_0)} \sum_{\mathfrak{t}_0 \in \mathcal{T}_0(\face, \gen)} \sum_{(e_1, \dots, e_\defect) \text{ good}} \frac{F(\mathrm{Contr}(\mathfrak{t}_0 ; e_1, \dots, e_\defect))}{\defect! \prod_{v \in \mathrm{Contr}(\mathfrak{t}_0 ; e_1, \dots, e_\defect)} \mathrm{Cat}(\deg v - 2)}.\]
Using the previous bounds on $G_\defect(\mathfrak{t}_0)$, we infer that there is a constant $C>0$ such that if conditional on $\Triv_0$, the vector $(E_1, \dots, E_\defect)$ has the uniform distribution on ordered good sets of edges of $\Triv_0$, then 
\begin{equation} \label{eq:RNexp}
C^{-\defect}
\le \frac{4^{-\defect} \binom{\lfloor  (\spar-2)/4\rfloor}{\defect}}{2^{-\defect} \binom{3 (\spar-2)}{\defect}}
\le \frac{\E[F(\Triv_{ \defect})]}{\E[F(\mathrm{Contr}(\Triv_0 ; E_1, \dots, E_\defect))]}
\le \frac{2^{-\defect} \binom{3 (\spar-2)}{\defect}}{4^{-\defect} \binom{\lfloor (\spar-2)/4\rfloor}{\defect}} 
\le C^\defect,
\end{equation}
as long as $ \defect \leq \spar/100$ say. This shows that a uniform random map in  $ \mathcal{T}_{\defect}( \face,\gen)$ is obtained by contracting $ \defect$ random edges in a random trivalent map in $ \mathcal{T}_{0}(\face, \gen)$ whose Radon--Nikodym derivative with respect to the uniform law on $\mathcal{T}_{0}(\face, \gen)$ is bounded from below by $C^{- \defect}$ and above by $C^{ \defect}$.

\subsection{The planar case}
\label{ssec:defaut_planar}

In the planar case with no defect $\defect=0$, the enumeration of trivalent plane maps, or by duality, of triangulations, goes back to~\cite{MNS70}, see also Krikun~\cite{Kri07}.
Specifically, trivalent plane maps with $\face$ faces are dual to plane triangulations with $\face$ vertices and therefore,
\begin{equation}\label{eq:triangulations}
\#\mathcal{T}_0(\face,0)
= \frac{2^{2\face-3} (3\face-6)!!}{(\face-1)! \, \face!!}
\equi[\face]  \frac{1}{72 \sqrt{6\pi}} \cdot  (12\sqrt{3})^\face \cdot \face^{-5/2}.
\end{equation}
Furthermore, it is known that the local limit of uniform random plane triangulations (the dual of trivalent plane maps) is given by the Uniform Infinite Planar Triangulation (of type 1) denoted by $ \mathbb{UIPT}$ below, see~\cite[Theorem~6.1]{Ste18}, as well as~\cite{AS03} for the pioneer works in the case of type 2 or 3 triangulations. This result has recently been extended to the case of essentially trivalent planar maps (i.e.~when the defect number is negligible in front of the size of the planar map) by Budzinski~\cite[Corollary~2]{Bud21}. Passing to the dual (denoted below with a $\dagger$ symbol), this implies in our context that, provided that $ \defect = o(  \face)$, we have 
\begin{equation} \label{eq:convUIPT}
\Triv_{ \defect}(\face, 0) \cvloi[\face]  \mathbb{UIPT}^\dagger,
\end{equation}
for the local topology. 
The next result is a control on the growth of $\#\mathcal{T}_ \defect(\face,0)$.

\begin{lemma}[Asymptotic enumeration of plane maps with defects]
\label{lem:ratio_defauts}
Uniformly for $\defect = o( \face)$, it holds that
\[\frac{2 \defect}{3 \face} \cdot \frac{\#\mathcal{T}_{\defect}(\face, 0)}{\#\mathcal{T}_{\defect-1}(\face,0)} \cv[\face] 1-\lambda_{\circ},\]
where $\lambda_{\circ} = 1 - \sqrt{3}/2$.\end{lemma} 

The quantity $\lambda_{\circ}$ is the asymptotic density of self-loops in large uniform trivalent plane maps, or more precisely the probability that the root vertex of the UIPT is a leaf, lying inside a single loop; equivalently, in the dual map, the root edge is itself a single loop. The value of $\lambda_{\circ}$ can be calculated by peeling, see e.g.~\cite[Section~2]{CLG19}.

The proof of Lemma~\ref{lem:ratio_defauts} is based on the contraction operation described in Section~\ref{ssec:contrac}.
For this it will be important to count the number of edges that we can actually contract. A key input estimate is a large deviation estimate on the number of loops in a uniform trivalent plane map due to Budzinski~\cite[Theorem~2]{Bud21}, namely for any $ \varepsilon>0$ there is some $\delta>0$ such that for all $ \face \geq 1$,
\[\P\left(\left| \frac{\#\{\text{loops in } \Triv_0(\face, 0)\}}{3 \face} - \lambda_{\circ}\right| \geq \varepsilon \right) \leq \e^{-\delta \face}.\]
Since $ \defect=o(\face)$ (and since the contraction of $ \defect$ edges may create at most $4 \defect$ loops in a trivalent map), we deduce from~\eqref{eq:RNexp} that a similar large deviation holds for $ \mathrm{T}_\defect(\face,0)$.
Equivalently, replacing $\varepsilon$ by $\varepsilon/(3(1-\lambda_{\circ}))$ and up to diminishing $ \delta$, this reads
\begin{equation}\label{eq:grande_dev_loops}
\#\left\{\mathfrak{t}_{\defect}  \in \mathcal{T}_\defect(\face, 0) : \left|\#\{\text{non loop edges in } \mathfrak{t}_\defect\} - 3\face(1-\lambda_{\circ})\right|  > \varepsilon\, \face \right\} \le \e^{-\delta \face} \# \mathcal{T}_\defect(\face, 0)
.\end{equation}

\begin{proof}[Proof of Lemma~\ref{lem:ratio_defauts}]
Recall that the contraction operation relates maps with defect $\defect$ and trivalent maps with an ordered list $(e_1, \dots, e_\defect)$ of ``good'' (i.e.~contractible) edges and recall precisely the formula~\eqref{eq:bijection_defauts}.
Observe that if $(e_1, \dots, e_\defect)$ is good, then so is $(e_1, \dots, e_{\defect-1})$. Let us denote by $e_\defect^-$ and $e_\defect^+$ the two extremities of $e_\defect$ in $\mathrm{Contr}(\mathfrak{t}_0 ; e_1, \dots, e_{\defect-1})$ and by $v_\defect$ the vertex of $\mathrm{Contr}(\mathfrak{t}_0 ; e_1, \dots, e_\defect)$ obtained by contracting $e_\defect$. 
Then taking $F=1$ in~\eqref{eq:bijection_defauts}, the contribution of each given trivalent map $\mathfrak{t}_0$ to the right-hand side equals
\[\sum_{(e_1, \dots, e_{\defect-1}) \text{ good}} \frac{1}{(\defect-1)! \prod_{v \in \mathrm{Contr}(\mathfrak{t}_0 ; e_1, \dots, e_{\defect-1})} \mathrm{Cat}(\deg v - 2)} \underbrace{\sum_{\substack{e_\defect \text{ such that}\\ (e_1, \dots, e_\defect) \text{ good}}}  \frac{\mathrm{Cat}(\deg e_\defect^- - 2) \mathrm{Cat}(\deg e_\defect^+ - 2)}{\defect\cdot\mathrm{Cat}(\deg v_\defect - 2)}}_{\displaystyle \eqqcolon H( \mathrm{Contr}( \mathfrak{t}_0 ; e_1, \dots, e_{ \defect-1}))}
.\]
Note that $\deg v_\defect = \deg e_\defect^- + \deg e_\defect^+ - 2$. Then each ratio in the definition of $H$ is uniformly bounded above, say by $K < \infty$, times $1/\defect$; moreover if $e_\defect$ is not incident to any $e_i$ for $i \le \defect-1$, then $\deg e_\defect^- = \deg e_\defect^+ = 3$ and $\deg v_\defect = 4$, in which case this ratio equals $1/(2\defect)$. The idea is to prove that $H( \mathrm{Contr}( \mathfrak{t}_0 ; e_1, \dots, e_{ \defect-1}))$ actually concentrates around $3\face(1-\lambda_\circ) / (2\defect)$. Indeed, given a trivalent map $\mathfrak{t}_0$ and a subset $(e_1, \dots, e_{\defect-1})$ of good edges, letting $A$ denote the set of edges different from the root, and which are not loops nor are adjacent to any $e_i$ for $i \le \defect-1$, and letting $B$ denote the set of edges which are incident to at least one $e_i$ for $i \le \defect-1$,
we have that
\[\frac{\# A}{2\defect} 
\le H( \mathrm{Contr}( \mathfrak{t}_0 ; e_1, \dots, e_{ \defect-1}))
\le \frac{\# A}{2\defect} + K \frac{\# B}{\defect}
\le \frac{\# A}{2\defect} + 4 K \frac{\defect-1}{\defect}
.\]
Notice in passing that the cardinal $\#A$ is actually a function $\#A(\mathrm{Contr}( \mathfrak{t}_0 ; e_1, \dots, e_{ \defect-1}))$ of the map after contractions. 
Using~\eqref{eq:bijection_defauts} twice, we then find
\begin{align*}
\# \mathcal{T}_\defect(\face, 0)
&= \sum_{\mathfrak{t}_0 \in \mathcal{T}_0(\face, 0)} \sum_{(e_1, \dots, e_{\defect-1}) \text{ good}} \frac{H( \mathrm{Contr}( \mathfrak{t}_0 ; e_1, \dots, e_{ \defect-1}))}{(\defect-1)! \prod_{v \in \mathrm{Contr}(\mathfrak{t}_0 ; e_1, \dots, e_{\defect-1})} \mathrm{Cat}(\deg v - 2)} 
\\
&= \sum_{ \mathfrak{t}_{ \defect-1} \in \mathcal{T}_{ \defect-1}(\face,0) } H( \mathfrak{t}_{\defect-1})
\\
&=  \sum_{ \mathfrak{t}_{ \defect-1} \in \mathcal{T}_{ \defect-1}(\face,0)} \left(\frac{\# A( \mathfrak{t}_{  \defect-1})}{2 \defect} + O(1)\right),
\end{align*}
where the $O(1)$ is uniform in $\face$ and $\defect$ as long as $ \mathrm{d}=o(\face)$.  By the large deviation estimate~\eqref{eq:grande_dev_loops} on the number of loops in $  \mathrm{T}_{ \defect}(\face,0)$ we have that $\#A$ is strongly concentrated around $3\face (1-\lambda_{\circ})$ and in particular 
\[\frac{\# \mathcal{T}_\defect(\face, 0)}{\# \mathcal{T}_{\defect-1}(\face, 0)} 
= \E\left[\frac{\# A( \Triv_{ \defect-1}(\face,0))}{2 \defect} + O(1)\right]
\equi[\face] \frac{ 3(1-\lambda_\circ) \face}{2 \defect},\]
as desired.
\end{proof}

\subsection{The unicellular case}
\label{ssec:unicellularestimates}

Let us now present the analogous results of the last section in the unicellular case. First, parallel to~\eqref{eq:triangulations} we have the following exact formula due to Lehman \& Walsh~\cite[Eq.~(9)]{WL72}:
\begin{equation}\label{eq:triangulations_high_gen}
\#\mathcal{T}_{0}(1,\gen) = \frac{2}{12^{\gen}} \frac{(6\gen-3)!}{\gen!(3\gen-2)!}
\equi[\gen] \frac{(12 \gen)^{2\gen} \e^{-2\gen}}{12 \sqrt{\pi \gen^3}}
.\end{equation}
It is well known that the local limit of large uniform trivalent graphs is given by the three-regular tree $  \mathbb{A}_3$.
We prove below that as soon as $ \defect =o( \gen)$ we have
\begin{equation} \label{eq:convthree-regular}
\Triv_{ \defect}(1, \gen) \cvloi[\gen]  \mathbb{A}_3,
\end{equation}
for the local topology. Finally, we prove the analogue of Lemma~\ref{lem:ratio_defauts} where in the unicellular case we have $\lambda_{\circ} =0$.

\begin{lemma}\label{lem:ratio_defauts_genus}
Uniformly for $\defect = o(\gen)$, it holds that
\[\frac{2 \defect}{6 \gen} \cdot \frac{\#\mathcal{T}_{\defect}(1,\gen)}{\#\mathcal{T}_{\defect-1}(1,\gen)} \cv[\gen] 1
.\]
\end{lemma}

Compared to the planar case (Lemma~\ref{lem:ratio_defauts}), notice here that the local limit of uniform essentially trivalent maps is a deterministic object and that the asymptotic density of loops in the local limit is $\lambda_\circ = 0$. Note also that the factor $3\face$ is now replaced by $6\gen$; in both cases, it corresponds (up to $3$) to $3\spar$ where $\spar=\face+2\gen$ is the sparsity parameter, which is (up to $O(\defect)$) the number of edges of the maps.
The proof of Lemma~\ref{lem:ratio_defauts_genus} is mutatis mutandis the same as in the planar case, we only need to replace appropriately the large deviation principle~\eqref{eq:grande_dev_loops}; this follows from~\eqref{eq:grandes_dev_high_gen} in the proof of~\eqref{eq:convthree-regular} below.

\subsubsection{Technical estimates via the configuration model} 

Let us recall the classical construction of random trivalent maps with $ \ver$ vertices using the configuration model. We let $ \ver$ be of the form $ \ver = 4 \gen -2$ for some $\gen \geq 1$, which corresponds to the number of vertices of a unicellular trivalent map with genus $\gen$.
Start with $ \ver$ tripods, i.e.~vertices having each three half-edges, hereafter called ``legs'', which are cyclically ordered, and list these legs from $1$ to $3\ver$. Note that $3 \ver$ is even since $ \ver = 4 \gen-2$ is, hence we can pair the legs using a uniform pairing over $\{1,2, \dots, 3 \ver\}$. Let us denote by $ \mathbf{P}_{\ver}$ the random labelled multi-graph obtained: it may have loops, multiple edges and may be disconnected. Furthermore, this graph has a cyclic orientation of the edges around each vertex coming from the tripods and so can be seen as a map (when it is connected) and we can speak of its faces. Moreover it is classical that conditionally on the event where $ \mathbf{P}_{\ver}$ is connected, it forms a uniform trivalent \emph{map} with $\ver$ vertices, with its half edges ordered from $1$ to $3\ver$, the first one being canonically the root. Note that this does not induce any bias since there are $3^{\ver-1} (\ver-1)!$ ways to label such a map whilst keeping the same cyclic ordering around each vertex.

We shall use the following large deviations result for the number of edges on small cycles in $ \mathbf{P}_{\ver}$, which follows by an application of the ``switching method'' from~\cite[Theorem 2.19]{Wor99}. 

\begin{lemma}[Large deviations for the small cycles]
\label{lem:grandes_dev_cycles_config}
For any $A>0$, Let $X_{ \ver}(A)$ be the number of edges of $ \mathbf{P}_{\ver}$ that belong to a non-backtracking cycle of length $\leq A$. For any $ \varepsilon>0$, there exists a constant $ \delta >0$ such that for any $\ver$ large enough we have 
\[\P( X_{ \ver}(A)  \geq \varepsilon \ver) \leq \mathrm{exp}(- \delta \ver).\]
\end{lemma}

\begin{proof}
Recall that the graph is trivalent so the number of cycles of length $\leq A$ passing through a given edge is bounded by $2^A$. Then the variations of $ X_{ \ver}(A)$ are bounded by some constant depending on $A$ if we switch two edges. By~\cite[Theorem~2.19]{Wor99} this implies the tail bounds of the claim for the deviations $|X_{ \ver}(A) - \E[X_{ \ver}(A)]|$. On the other hand, the expectation of $X_{ \ver}(A)$ is converging (see~\cite[Eq.~(8) in Section~2.3]{Wor99}), so it is bounded; the claim then follows.
\end{proof}

\subsubsection{Proof of the local convergence and the ratios}

Let us start by proving that the local limit of $ \Triv_{\defect}(1, \gen)$ is the infinite three-regular tree $ \mathbb{A}_3$. 

\begin{proof}[Proof of~\eqref{eq:convthree-regular}]
Recalling that $\ver=4\gen-2$, as well as the exact formula~\eqref{eq:triangulations_high_gen}, the probability that the random graph $ \mathbf{P}_{\ver}$ is connected and unicellular is
\[{3^{\ver-1} (\ver-1)! \#\mathcal{T}_{0}(1,\gen)} \cdot  \left( { \frac{(3\ver)!}{(3\ver/2)! 2^{3\ver/2}}} \right)^{-1}
\equi[\gen] \frac{2}{3\ver}.\]
Consequently the large deviation estimates from Lemma~\ref{lem:grandes_dev_cycles_config} for the number of edges belonging to small cycles also hold for  $ \Triv_{0}(1, \gen)$ for $\ver = 4 \gen-2$ up to changing $\delta >0$. 
Now fix $A>0$ and observe that any vertex at distance more than $ A$ from a cycle of length $ \leq A$ has the same $A$-neighborhood as the origin in the three-regular tree. Since we are dealing with trivalent graphs, there are at most $ 2^{A}C$ vertices within distance $A$ of a given subset of $C$ vertices. In particular, we deduce the following concentration inequality in $  \Triv_{ 0}(1, \gen)$. Let $ N_{A}(G)$ be the number of vertices in the graph $G$ whose $A$-neighborhood is tree-like, then for any $\varepsilon>0$, there exists $\delta >0$ such that for all $\gen$ large enough
\[\P( |N_{A}(\Triv_{ 0}(1, \gen)) - 4\gen| \geq \varepsilon \gen) \leq \e^{-\delta \gen}.\]
We aim for the same bound for $\Triv_{\defect}(1, \gen)$ when $ \defect = o(\gen)$ and argue as in the last section. Recall $ \Triv_{ \defect}(1, \gen)$ is obtained by contracting $ \defect$ edges from a (non-uniform) random map in $ \mathcal{T}_0(1, \gen)$ whose law has a  Radon--Nikodym derivative bounded above by $ \mathrm{C}^{ \defect}$ with respect to the uniform law. Since $  \defect = o( \gen)$ and since there are fewer than $2^{A}$ edges within distance $A$ of the contracted edges, the above display still holds for $  \Triv_{ \defect}(1, \gen)$, with possibly a smaller $\delta$, which does not depend on $\defect$, namely: for any $\varepsilon>0$, there exists $\delta >0$ such that for all $\gen$ large enough, for $\defect/\gen$ small enough,
\begin{equation}\label{eq:grandes_dev_high_gen}
\P( |N_{A}(\Triv_{\defect}(1, \gen)) - 4\gen| \geq \varepsilon \gen) \leq \e^{-\delta \gen}.
\end{equation}
By invariance by re-rooting, the local limit~\eqref{eq:convthree-regular} follows.
\end{proof}

We may finally prove Lemma~\ref{lem:ratio_defauts_genus}.

\begin{proof}[Proof of Lemma~\ref{lem:ratio_defauts_genus}]
Taking $A = 1$ in~\eqref{eq:grandes_dev_high_gen} shows a large deviation principle for the number of non loop edges in $\Triv_{\defect}(1, \gen)$, whose proportion concentrates around $1-\lambda_\circ = 1$.
The proof of Lemma~\ref{lem:ratio_defauts_genus} is now mutatis mutandis the same as in the planar case, replacing~\eqref{eq:grande_dev_loops} by this estimate.  
\end{proof}

\section{On the size of the core and kernel}
\label{sec:preuve_core_kernel}

With the enumerations lemmas in place, we can proceed to the proofs of Theorem~\ref{thm:defauts_ker} and~\ref{thm:aretes_core}.

\subsection{Number of defects of the kernel}
\label{ssec:defauts_kernel}

Let us begin with the proof of Theorem~\ref{thm:defauts_ker}; recall the notation from the statement. 

\begin{proof}[Proof of Theorem~\ref{thm:defauts_ker}]
Let $\spar_n = \face_n+2\gen_n$.
Recall from Proposition~\ref{prop:loi_core_kernel} that,
\[\P\left(\Ker(\Map_n(\face_n,\gen_n)) \textrm{ has defect } \defect\right)
=  \frac{\# \mathcal{T}_\defect(\face_n,\gen_n) \cdot \Phi_{n}( 3\spar_n-\defect-6 )}{\# \mathfrak{M}_n(\face_n,\gen_n)}
\]
for any $\defect \in\{0, \dots, 2\spar_n-5\}$, whereas it equals $0$ otherwise since in this case $\mathcal{T}_\defect(\face_n,\gen_n)$ is empty.
The idea is to consider the ratios of the numerator evaluated at $\defect+1$ and at $\defect$ to find the optimal value $\defect$ which maximise this quantity, and control the deviations for other values of $\defect$.
Since $ 1 \ll \spar_{n} \ll n$, then, using Lemmas~\ref{lem:phi_et_Phi},~\ref{lem:ratio_defauts}, and~\ref{lem:ratio_defauts_genus} at the second line, uniformly for $ \defect \ll \spar_{n}$,
\begin{align}
\frac{\P\left(\Ker(\Mapb_n^{\spar_n}) \textrm{ has defect } \defect+1\right)}{\P\left(\Ker(\Mapb_n^{\spar_n}) \textrm{ has defect } \defect\right)}
&\quad=\enskip \frac{\Phi_{n}( 3\spar_n-(\defect+1)-6 )}{\Phi_{n}( 3\spar_n-\defect-6 )} \times \frac{ \# \mathcal{T}_{ \defect +1}(\face_n,\gen_n)}{{ \# \mathcal{T}_{ \defect}(\face_n,\gen_n)}}
\notag\\ 
&\equi  \sqrt{\frac{2(3\spar_n - (\defect+1) - 6)}{n}}  \times \frac{3 \spar_n(1-\lambda_{\circ})}{2 (\defect+1)}
\notag\\
&\equi \sqrt{\frac{6\spar_n}{n}} \frac{3 \spar_n(1-\lambda_{\circ})}{2 (\defect+1)}
= \frac{D_n}{\defect+1}, \label{eq:ratio_estimee}
\end{align}
where we recall that $D_{n} = 3 (1-\lambda_{\circ}) \sqrt{\frac{3}{2} \frac{\spar_n^3}{n}}$. It is therefore natural to expect that the defect number concentrates around $D_{n}$ and we now make this precise in the two regimes.
\medskip

We now turn to the first statement, when $n^{-1/3} \spar_n \to a \in [0,\infty)$. Note that in this implies that $D_n \to 3 (1-\lambda_{\circ}) \sqrt{3 a^3 / 2}$; let us write this limit as $c \ge 0$.
By induction~\eqref{eq:ratio_estimee} implies that for any $\defect\ge0$, it holds
\[\frac{\P(\Ker( \Mapb_n^{\spar_n}) \textrm{ has defect } \defect)}{\P(\Ker( \Mapb_n^{\spar_n}) \textrm{ has defect } 0)}
\cv \frac{c^\defect}{\defect!}
.\]
In order to conclude, it remains to show that 
\begin{equation}\label{eq:goalpoisson}
\frac{\P(\Ker( \Mapb_n^{\spar_n}) \textrm{ has defect}  \geq D)}{\P(\Ker( \Mapb_n^{\spar_n}) \textrm{ has defect } 0)}
\end{equation}
can be made arbitrarily small (uniformly in $n$ and $ \mathrm{s}_{n}$) provided that $D$ is large enough. To see that, note that from~\eqref{eq:crude1} we have $\# \mathcal{T}_{\defect}(\face,\gen) \le \frac{(3\spar/2)^\defect}{\defect!} \# \mathcal{T}_{0}(\face,\gen)$.
On the other hand by Lemma~\ref{lem:phi_et_Phi} there exists a constant $K>0$ such that for every $n$ large enough (so e.g. $3\spar_n \le n/32$),
\[\sup_{\defect \le 3\spar_n-6} \sqrt{\frac{n}{\spar_n}} \cdot \frac{\Phi_{n}( 3\spar_n-(\defect+1)-6 )}{\Phi_{n}( 3\spar_n-\defect-6 )}
\le K/2.\]
Therefore for every $n$ large enough, for every $\defect \le 3\spar_n-6$ it holds
\[\frac{\Phi_{n}( 3\spar_n-\defect-6 )}{\Phi_{n}( 3\spar_n-6 )} \le \left(K \sqrt{\frac{\spar_n}{n}}\right)^\defect.\]
Combining the two bounds yields for every $n$ large enough, for every $\defect \le 3\spar_n-6$,
\[\frac{\P(\Ker( \Mapb_n^{\spar_n}) \textrm{ has defect } \defect)}{\P(\Ker( \Mapb_n^{\spar_n}) \textrm{ has defect } 0)}
= \frac{\#\mathcal{T}_\defect(\face_n,\gen_n) \cdot \Phi_{n}(3\spar_n-\defect-6)}{\# \mathcal{T}_0(\face_n,\gen_n)  \cdot \Phi_{n}(3\spar_n-6 )}
\le \frac{1}{\defect!} \left(\frac{3K}{2} \sqrt{\frac{\spar_n^3}{n}}\right)^\defect 
.\]
Recall that we assume that $\spar_n^3/n$ has a finite limit, so this sequence is uniformly bounded. We easily deduce that~\eqref{eq:goalpoisson} can be made arbitrary small provided that $D$ is large enough and this concludes the convergence to a Poisson distribution.
\medskip

Let us now turn to the second regime, where $n^{-1/3} \spar_n \to \infty$ (but still $n^{-1} \spar_n \to 0$) and let us replace for convenience $D_n$ by its integer part; note that $D_n = o(\spar_n)$. 
Recall the asymptotic behaviour~\eqref{eq:ratio_estimee} valid uniformly for $\defect = o(\spar_n)$. Fix any $\varepsilon > 0$ and let $n$ be large enough so, for any $k \in [2\varepsilon D_n, \varepsilon^{-1} D_n]$,
\[\frac{\P(\Ker(\Mapb_n^{\spar_n}) \textrm{ has defect } D_n+k)}{\P(\Ker(\Mapb_n^{\spar_n}) \textrm{ has defect } D_n)}
\le \prod_{j=1}^k (1+\varepsilon) \frac{D_n}{D_n+j}
\le \left(\frac{1+\varepsilon}{1+2\varepsilon}\right)^k
.\]
In addition, by~\eqref{eq:crude1} and Lemma~\ref{lem:phi_et_Phi} there exists a constant $K>0$ such that for every $n$ large enough (so e.g. $3\spar_n \le n/32$) and any $\defect$,
\[\frac{\P(\Ker(\Mapb_n^{\spar_n}) \textrm{ has defect } \defect+1)}{\P(\Ker(\Mapb_n^{\spar_n}) \textrm{ has defect } \defect)}
\le K \frac{D_n}{\defect+1} 
.\]
Thus, similarly, for $k \ge \varepsilon^{-1} D_n$,
\[\frac{\P(\Ker(\Mapb_n^{\spar_n}) \textrm{ has defect } D_n+k)}{\P(\Ker(\Mapb_n^{\spar_n}) \textrm{ has defect } D_n)}
\le \left(\frac{K}{1+\varepsilon^{-1}}\right)^k
.\]
Taking $\varepsilon>0$ small enough, we infer that there exists $\delta > 0$ such that for every $n$ large enough,
\[\frac{\P(\Ker(\Mapb_n^{\spar_n}) \textrm{ has defect} \ge (1+2\varepsilon) D_n)}{\P(\Ker(\Mapb_n^{\spar_n}) \textrm{ has defect } D_n)}
\le \sum_{k \ge 2\varepsilon D_n} (1-\delta)^k = \delta^{-1} (1-\delta)^{2\varepsilon D_n}
.\]
The negative deviations are treated similarly and are left to the reader.
\end{proof}

\subsection{Asymptotic enumeration}
\label{ssec:asymp_enum}
In return the probabilistic estimates in the preceding proof can be turned into asymptotic estimates on the number trivalent maps with defects. In the regime $\spar_{n} = O(n^{1/3})$ we obtain actual asymptotics.

\begin{proof}[Proof of Corollary~\ref{cor:enumeration}]
Suppose that $n^{-1/3} \spar_n \to a \ge 0$ and that either $(\face_n, \gen_n) = (\spar_n, 0)$ or $(\face_n, \gen_n) = (1, (\spar_n-1)/2)$. Then the convergence of  $ \mathbb{P}( \mathrm{Ker}( \Mapb_{n}^{\spar_{n}}) \mbox{ has zero defect})$ towards the probability that a Poisson law is equal to $0$ can be written by Proposition~\ref{prop:loi_core_kernel} as
\[\frac{\# \mathcal{T}_0(\face_n,\gen_n) \cdot \Phi_{n}( 3\spar_n-6 )}{\# \mathfrak{M}_n(\face_n,\gen_n)}
\cv \exp\left(- (1-\lambda_{\circ}) \sqrt{(3a)^3 / 2}\right)
,\]
where we recall that $\lambda_\circ$ equals $1 - \sqrt{3}/2$ when $\gen_n=0$ and $0$ when $\face_n=1$.
The asymptotic behaviour of $\# \mathcal{T}_0(\face_n,\gen_n)$ when $\gen_n=0$ has been recalled in~\eqref{eq:triangulations}, whereas when $\face_n=1$ it is given by~\eqref{eq:triangulations_high_gen}. Also the behaviour of $\Phi_{n}( 3\spar_n-6 )$ follows from Corollary~\ref{cor:equivPhikn}, namely, with $\gamma = (3a)^3$,
\begin{align*}
\Phi_{n}(3 \spar_{n}-6)
& \equi \frac{\e^{- \sqrt{\gamma/2}}}{2\sqrt{2 \pi (3\spar_{n}-6)}} \cdot 4^{n}  \cdot \left(  \frac{\e n}{2 (3\spar_{n}-6)}\right)^{ \frac{3\spar_{n}-6}{2}}
\\
& \equi \frac{6^3 \e^{- \sqrt{\gamma/2}}}{2\sqrt{6 \pi}} \cdot 4^{n}  \cdot \left(  \frac{\e n}{6 \spar_{n}}\right)^{3\spar_{n}/2}  \cdot \spar_{n}^{5/2} \cdot n^{-3}
.\end{align*}

Combined with~\eqref{eq:triangulations} we derive the asymptotic formula for the number of plane maps with $n$ edges, $\face_n$ faces, and a trivalent kernel: when both $\gen_n=0$ and $n^{-1/3} \face_n \to f$, then $n^{-1/3} s_{n} \rightarrow f$, so taking $\gamma=(3f)^{3}$ we get
\[\#\mathcal{T}_0(\face_n,0) \cdot \Phi_{n}(3 \face_{n}-6)
\equi  \frac{\e^{- \sqrt{\gamma/2}}}{4 \pi} \cdot n^{-3} \cdot 4^{n}  \cdot \left( 2^{1/3} \frac{\e n}{\face_{n}}\right)^{3\face_{n}/2} 
.\]

Similarly, using~\eqref{eq:triangulations_high_gen} instead, we obtain the asymptotic formula for the number of unicellular maps with $n$ edges, genus $\gen_n$ faces, and a trivalent kernel: when both $\face_n=1$ and $n^{-1/3} g_{n} \rightarrow g$, then $n^{-1/3} s_{n} \rightarrow 2g$, so taking $\gamma=(6g)^{3}$ we get
\[\#\mathcal{T}_0(1,\gen_n) \cdot \Phi_{n}(6 \gen_{n}-6)
\equi \frac{ \e^{- \sqrt{\gamma/2}}}{2\pi } \cdot \gen_n^{-1/2} \cdot n^{-3/2} \cdot 4^{n}  \cdot \left(  \frac{\e n^3}{12 \gen_n}\right)^{\gen_n}
.\]
Our claim then follows from the convergence in the beginning of this proof.
\end{proof}

\subsection{Volume of the core}
\label{ssec:vol_core}

Theorem~\ref{thm:aretes_core} now easily follows from our previous results.

\begin{proof}[Proof of Theorem~\ref{thm:aretes_core}]
Since $\#\mathrm{Edges} (\Ker( \Mapb_n^{\spar_{n}}))=3s_{n}-\Defect(\Ker(\Mapb_n^{\spar_{n}})) -6$, we infer from Theorem~\ref{thm:defauts_ker} that as $n\to\infty$, the number of edges of $\Ker( \Mapb_n^{\spar_{n}})$ concentrates around $K_{n}$ defined by
\[K_n = 3\spar_n - 3 (1-\lambda_{\circ}) \sqrt{\frac{3}{2} \frac{\spar_n^3}{n}} \equi 3\spar_n.\]
Recall from Proposition~\ref{prop:loi_core_kernel} the conditional law of the number of edges of the core given the kernel; then Corollary~\ref{cor:TCL_core} implies that
\[\frac{\#\mathrm{Edges}(\Core( \Mapb_n^{\spar_{n}}))}{C_n} \cvproba 1,
\qquad\text{where}\qquad
C_n = \frac{K_n + \sqrt{K_n^{2} + 8 n K_n }}{4}
\equi \sqrt{\frac{3 n \spar_n}{2}},
\]
from which our claim follows.
\end{proof}

Let us define $C_n'$ as $C_n$ in the preceding display but using the random number of edges of the kernel instead of the deterministic quantity $K_n$. Then one can check that $4(C_n' - C_n)$ is asymptotically equivalent to $\#\mathrm{Edges}(\Ker(\Mapb_n^{\spar_{n}}))-K_n$. By Theorem~\ref{thm:defauts_ker} this is much smaller than $K_n$, and thus much smaller than $\spar_n$, therefore, when $\spar_n = O(\sqrt{n})$, we can deduce an unconditioned CLT from Corollary~\ref{cor:TCL_core}, namely 
\[\frac{2}{\sqrt{n}} \cdot  \left( \#\mathrm{Edges}(\Core( \Mapb_n^{\spar_{n}})) - C_n \right) \cvloi N(0,1).\]
When $\spar_n \gg \sqrt{n}$ however, one would need a tighter control on the fluctuations of the size of the kernel. 
Although our estimates are not precise enough, we believe that the following extensions hold. 

\begin{conjecture}\label{conjecture_TCL}
Assume that $n^{-1/3} \spar_n \to \infty$ and $n^{-1} \spar_n \to 0$ and let $\lambda_\circ$ as in~\eqref{eq:lambda}. Define
\[D_n = 3 (1-\lambda_{\circ}) \sqrt{\frac{3}{2} \frac{\spar_n^3}{n}},
\qquad
K_n = 3\spar_n - D_n,
\qquad\text{and}\qquad
C_n = \frac{K_n + \sqrt{K_n^{2} + 8 n K_n }}{4}.\]
Then
\[\frac{\Defect(\Ker(\Mapb_n^{\spar_{n}})) - D_n}{\sqrt{D_n}} \cvloi N(0,1),\]
where $N(0,1)$ is the standard Gaussian distribution. Consequently,
\[\frac{2}{\sqrt{n}} \cdot  \left( \#\mathrm{Edges}(\Core( \Mapb_n^{\spar_{n}})) - C_n \right) \cvloi N(0,1).\]
\end{conjecture}


\section{Scaling limits for the attached trees}
\label{sec:contour}

In this section we establish the scaling limits for the trees attached to the core of the random map $\Map_n(\face_n, \gen_n)$. 
We show that this forest converges after normalisation by the factor $ \sqrt{n/\spar_n}$ towards a forest coded in the usual way by a Brownian motion with negative drift for which we review the excursion theory. The results are used in the next section when proving Theorem~\ref{thm:intermediate}.

\subsection{Excursions theory for linear Brownian motion with drift}
\label{sec:excursionBMD}
 
 Let $ \mathcal{B}$ denote a standard Brownian motion started from $0$. Fix $\kappa >0$ (in our application below we  shall take $\kappa = \sqrt{3/2}$)  and let us consider the Brownian motion with linear drift $-\kappa$ and its running infimum process defined respectively for any $t \ge 0$ by 
\[\mathcal{B}^{\kappa}_{t} = \mathcal{B}_{t}-\kappa t
\qquad\text{and}\qquad
\underline{\mathcal{B}}^{\kappa}_{t} = \inf_{0 \leq s\leq t} \mathcal{B}^{\kappa}_{s}
.\]
Let us now recall the excursion theory of $ \mathcal{B}^{\kappa}$ above its minimum. 
First, when $\kappa=0$, it is well know that the excursions of a Brownian motion form a Poisson process with local time given by $\underline{\mathcal{B}}^{0}$ and with intensity given by the It\={o} excursion measure $\mathbf{n}$, see e.g.~\cite{LG10} or~\cite[Chapter XII]{RY99}. With our normalisation, the measure $ \mathbf{n}$ can be  disintegrated as
\[\mathbf{n}(\cdot) = \int_0^\infty \mathbf{n}_a(\cdot) \frac{\d a}{\sqrt{2\pi a^3}},\]
where $\mathbf{n}_a$ denotes the law of the Brownian excursion with duration $a$. 

When $\kappa>0$, by Girsanov's formula, the process $\mathcal{B}^\kappa$ is absolutely continuous with respect to $\mathcal{B}$, with density given by the exponential martingale $(\exp(-\kappa \mathcal{B}_t - \kappa^2t/2) ; t \ge 0)$. Using this and the exponential formula for Poisson random measures, we deduce that the excursions of $ \mathcal{B}^{\kappa}$ above its running infimum $\underline{\mathcal{B}}^{\kappa}$ (still using $\underline{\mathcal{B}}^{\kappa}$ as local time) are again distributed as a Poisson process with excursion intensity given by 
\[\mathbf{n}^{\kappa}(\cdot) = \int_0^\infty \mathbf{n}_a(\cdot) \frac{\exp(-\kappa^2 a/2)}{\sqrt{2\pi a^3}} \d a.
\]
Finally, let us describe the so-called Bismut decomposition for the excursion measure $ \mathbf{n}^{\kappa}$. For this, we introduce  the size-biased excursion measure $\overline{ \mathbf{n}}^{\kappa}$ on excursions $ \mathbf{e}$ of duration $\zeta( \mathbf{e})$ together with a distinguished time $u \in [0, \zeta( \mathbf{e})]$ by
\begin{equation}\label{eq:excursion_biaisee}
\overline{ \mathbf{n}}^{\kappa}( \d  \mathbf{e}, \d u) = \int_0^\infty  \d a \frac{\exp(-\kappa^2 a/2)}{\sqrt{2\pi a^3}} \mathbf{n}_a( \d  \mathbf{e}) 1_{[0,a]}(u) \d u.
\end{equation}
Contrary to $  \mathbf{n}^{\kappa}$, the above measure has finite total mass, so it can be used to define a probability distribution after normalisation; precisely,
\[\int \overline{\mathbf{n}}^{\kappa}( \d  \mathbf{e}, \d u) = \int_0^\infty  \d a \frac{\exp(-\kappa^2 a/2)}{\sqrt{2\pi a}} = \frac{1}{\kappa}.\]
For an excursion $ \mathbf{e}=(\mathbf{e}_t ; 0 \le t \le \zeta( \mathbf{e}))$ and a time $u \in (0, \zeta( \mathbf{e}))$, let $ \mathbf{e}^{u,-} = (\mathbf{e}_{u-t} ; 0\le t \le u)$ and $ \mathbf{e}^{u,+} = (\mathbf{e}_{u+t} ; 0\le t \le \zeta-u)$; let also $\mathcal{B}^-$ and $\mathcal{B}^+$ be two independent Brownian motions both started from $x>0$ under $\P_x$ and both stopped when hitting $0$, at time $T^-$ and $T^+$ respectively. Let $F$ be a continuous and bounded function, the so-called Bismut decomposition~\cite[Theorem~4.7, Chap~XII, p502]{RY99} reads
\[\int \overline{\mathbf{n}}^{0}( \d  \mathbf{e}, \d s) F\big( \mathbf{e}^{s,-},  \mathbf{e}^{s,+}\big)
= 2 \int_0^\infty \E_x\left[F\big(\mathcal{B}^-, \mathcal{B}^+\big)\right] \d x.\]
In the case $\kappa >0$ we can write with obvious notation
\begin{align*}
\int \kappa \overline{\mathbf{n}}^{\kappa}( \d  \mathbf{e}, \d u) F\big( \mathbf{e}^{u,-},  \mathbf{u}^{s,+}\big)
&= \int \kappa \overline{\mathbf{n}}^{0} ( \d  \mathbf{e}, \d u) \e^{-\kappa^2\zeta( \mathbf{e})/2} F \big( \mathbf{e}^{u,-},  \mathbf{e}^{u,+}\big)  \d  u
\\
&= \int_0^\infty 2\kappa \cdot \E_x\left[\e^{-\kappa^2 (T^-+T^+)/2} F\big(\mathcal{B}^-, \mathcal{B}^+\big)\right] \d x
\\
&= \int_0^\infty 2\kappa \e^{-2 \kappa x} \, \E_x\left[F\big(\mathcal{B}^{\kappa,-}, \mathcal{B}^{\kappa,+}\big)\right] \d x
.\end{align*}
In words, if one first samples an excursion $ \mathbf{e}$ and a time $u$ from the law $\kappa \overline{\mathbf{n}}^{\kappa}$, then the random variable $\mathbf{e}_u$ follows an exponential law with rate $2 \kappa$ and conditionally given this value, the time-reversed past and the future of the excursion are two independent Brownian motions with drift $-\kappa$ started from $\mathbf{e}_u$ and killed upon hitting $0$, see Figure~\ref{fig:bismut} for a pictorial representation.

\begin{figure}[!ht]\centering
\includegraphics[width=10cm]{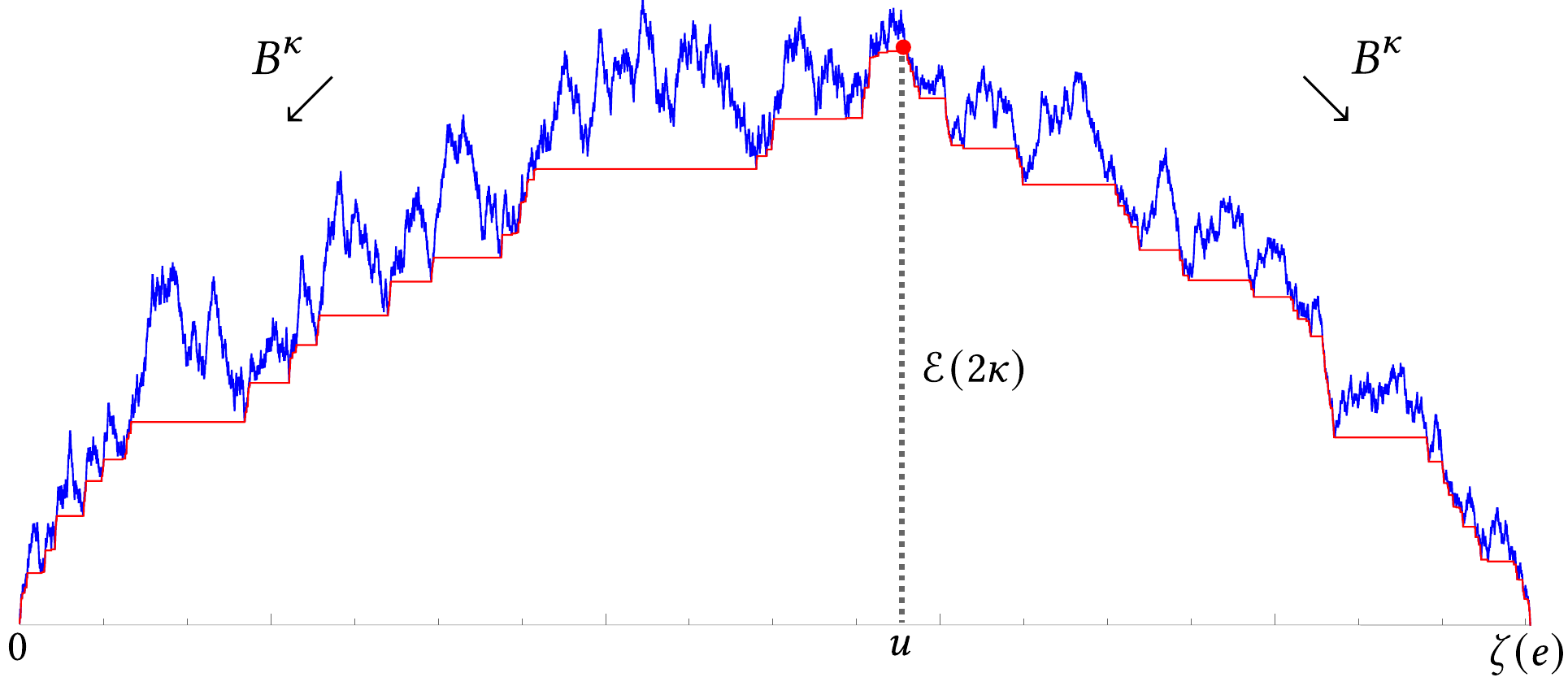}
\caption{The Bismut decomposition of $ \overline{  \mathbf{n}}^{\kappa}$.}
\label{fig:bismut}
\end{figure}

Let us now make a connection with another appearance of the exponential law with rate $2\kappa$. Indeed, recall that $(\exp(2\kappa \mathcal{B}_{t}^\kappa) ; t \ge 0)$ is a martingale, so a classical application of the optional stopping theorem shows that $\sup_{t \geq 0} \mathcal{B}^{\kappa}_{t}$ follows this very exponential law.
The link can be done via Bismut's decomposition. Indeed, let us extend $\mathcal{B}^\kappa$ to the negative half-line letting $(-\mathcal{B}^\kappa_{-t})_{t \ge 0}$ be an independent copy of $\mathcal{B}^\kappa$. Then $-\inf_{t \le 0} \mathcal{B}^\kappa$ follows this exponential law. Let $J^- < 0$ denote the (a.s.~unique) time such that $\mathcal{B}^\kappa_{J^-} = \inf_{t \le 0} \mathcal{B}^\kappa$ and let $J^+ = \inf\{t>0 : \mathcal{B}^\kappa_t = \inf_{t \le 0} \mathcal{B}^\kappa\}$. It follows from Bismut's decomposition that the pair $((\mathcal{B}^\kappa_t)_{t \in [J^-,J^+]}, -J^-)$ has the law $\kappa \overline{\mathbf{n}}^\kappa$. 

Also, notice that conditional on $\mathcal{B}^\kappa_{J^+}$, the path after time $J^+$ is then an independent copy of $\mathcal{B}^\kappa$ starting from this value, so its excursions are described by the infinite measure $\mathbf{n}^\kappa$.
Finally, we define the path $\mathcal{W}^\kappa = (\mathcal{W}^\kappa_t)_{t\ge 0}$ by
\begin{equation}\label{eq:excursion_brownien_drift}
\mathcal{W}^\kappa_t = \mathcal{B}^\kappa_{t+J^-} - \mathcal{B}^\kappa_{J^-},
\end{equation}
which therefore starts with a size-biased excursion before evolving like $\mathcal{B}^\kappa$. 

Let us mention~\cite[Section~7.7.7]{Pit06} and~\cite{Jan05} for related discussions with similar objects.

\subsection{Scaling limit of the random forest}

We now aim at showing that $\mathcal{W}^\kappa$ just defined in~\eqref{eq:excursion_brownien_drift}  is the scaling limit of the contour of the random forest attached to the core in order to recover the random map $\Map_n(\face_n, \gen_n)$.
Let us argue conditionally given the core and its number of edges, say $c_n$, which, in the framework of Theorem~\ref{thm:aretes_core}, is typically of order $\sqrt{n \spar_n}$, which is both much larger than $\sqrt{n}$ and much smaller than $n$.

Recall from Section~\ref{sec:decomposition} and especially Figure~\ref{fig:contour} that we actually consider a forest with a mark which is either an oriented edge in the first tree, or an extra symbol to mean that we keep the root edge of the core. We let $W^n$ denote the contour of this forest, $A^n$ its first hitting time of $-1$, which codes the size of the first tree, and let $R^n$ have the uniform distribution on $\{0, \dots, A^n-1\}$, which codes for the position of the root edge of the map. Then the pair $(W^n, R^n)$ has the uniform distribution on the first-passage paths with $\pm1$ increments, which end at time $2n$ by hitting $-2c_n$ for the first time, together with an instant smaller than their first hitting time of $-1$. 
The main result of this section is the following.

\begin{proposition}\label{prop:cv_contour_brownien_drift}
If $\sqrt{n} \ll c_n \ll n$, then for any $\kappa >0$, the convergence in distribution
\[\left(\frac{c_{n}}{\kappa n} W^{n}_{(\kappa n/c_{n})^2 t}\right)_{t \ge 0}
\cvloi \left(\mathcal{W}^{\kappa}_{t}\right)_{t \ge 0},
\]
holds in the uniform topology on compact intervals, where $\mathcal{W}^\kappa$ is defined in~\eqref{eq:excursion_brownien_drift}.
\end{proposition}

Let $B^n$ denote a uniform random $\pm1$ path starting from $0$ at time $0$ and ending at $-2c_n$ at time $2n$ and recall that $W^n$ is obtained from $B^n$ by a cyclic shift at the first time the latter reaches its overall minimum.
The path $B^n$ is quite simple and well-known and when $c_n=o(n)$ a global convergence of $B^n$ to the Brownian bridge has been established, see e.g.~\cite[Theorem~20.7]{Ald85}.
However here, we are interested in the behaviour of this path viewed in a smaller time scale, around both the starting point and the endpoint. To this end, let us set $\widehat{B}^{n}_{i} = -2c_{n}-B^{n}_{2n-i}$ for every $i \in \{0, \dots, 2n\}$. Note that $\widehat{B}^{n}$ has the same law as $B^{n}$.
On a time-scale which is small compared to $n$, the recentered paths do not feel the bridge conditioning and the fluctuations simply converge to Brownian motions.

\begin{lemma}\label{lem:cv_recentre_brownien}
Suppose that $c_n = o(n)$. Let $N_{n}\to \infty$ be such that $N_{n} = o(n)$. Let $ \mathcal{B}$ and $\widehat{ \mathcal{B}}$ be two independent Brownian motions. 
The convergence in distribution
\[\left(\frac{1}{\sqrt{N_{n}}} \left(B^{n}_{N_{n}t}+\frac{c_{n} N_{n}}{n}t, \widehat{B}^{n}_{N_{n}t}+\frac{c_{n} N_{n}}{n}t\right)\right)_{t \ge 0}
\cvloi \left( \mathcal{B}_{t}, \widehat{ \mathcal{B}}_{t}\right)_{t \ge 0},\]
holds in the uniform topology on compact intervals.
\end{lemma}

\begin{proof}
Fix $T>0$ and suppose that $n$ is large enough so $N_n T < n$. Let $S^{n}$ denote the asymetric random walk with step distribution $\P(S^{n}_{1} = -1) = 1-\P(S^{n}_{1} = 1) = \frac{1}{2}(1+\frac{c_{n}}{n})$. Note that $B^{n}$ has the law of $S^{n}$ conditioned to $S^n_{2n} = -2c_{n}$ and further that $\E[S^{n}_{2n}] = -2c_{n}$. Let us denote by $\widehat{S}^{n}$ and independent copy of $S^{n}$.
Then the Markov property yields the following absolute continuity relation: For any continuous and bounded function $F$, we have
\begin{multline*}
\E\left[F\left(\left(\frac{1}{\sqrt{N_{n}}} \left(B^{n}_{N_{n}t}+\frac{c_{n} N_{n}}{n}t, \widehat{B}^{n}_{N_{n}t}+\frac{c_{n} N_{n}}{n}t\right)\right)_{0 \le t \le T}\right) \right]
\\
= \E\left[F\left(\left(\frac{1}{\sqrt{N_{n}}} \left(S^{n}_{N_{n}t}+\frac{c_{n} N_{n}}{n}t, \widehat{S}^{n}_{N_{n}t}+\frac{c_{n} N_{n}}{n}t\right)\right)_{0 \le t \le T}\right) \frac{\P(S^{n}_{2n-2N_{n}T} = -(S^{n}_{N_{n}T}+\widehat{S}^{n}_{N_{n}T}+2c_{n}))}{\P(S^{n}_{2n} = -2c_{n})}\right]
.\end{multline*}
We first claim that $(N_{n}^{-1/2} (S^{n}_{N_{n}t}+\frac{c_{n} N_{n}}{n}t, \widehat{S}^{n}_{N_{n}t}+\frac{c_{n} N_{n}}{n}t) ; 0 \le t \le T)$ converges in distribution to the pair $((\mathcal{B}_{t}, \widehat{\mathcal{B}}_{t}) ; 0 \le t \le T)$. By e.g.~\cite[Theorem~16.14]{Kal02} it suffices to prove the convergence at time $t=1$ and the latter easily follows by considering the characteristic function.
By Skorokhod's representation, let us assume that this convergence holds almost surely. We next control the ratio of probabilities in the absolute continuity relation.
By Stirling's formula, as $n\to\infty$, using also that $c_{n} = o(n)$ in the last line, we obtain after straightforward calculations:
\begin{align*}
\P(S^{n}_{2n} = -2c_{n}) 
&= \left(\frac{n+c_{n}}{2n}\right)^{n+c_{n}} \left(\frac{n-c_{n}}{2n}\right)^{n-c_{n}} \frac{(2n)!}{(n-c_{n})! (n+c_{n})!}
\\
&\sim \frac{1}{\sqrt{2\pi}} \sqrt{\frac{2n}{(n-c_{n}) (n+c_{n})}} \sim \frac{1}{\sqrt{n\pi}}
.\end{align*}
Similarly, since both $N_{n}, c_{n} = o(n)$, then
\[\P(S^{n}_{2n-2N_{n}T} = -(S^{n}_{N_{n}T}+\widehat{S}^{n}_{N_{n}T}+2c_{n}))
\sim \frac{1}{\sqrt{(n-N_{n}T)\pi}}
\sim \frac{1}{\sqrt{n\pi}}
.\]
We conclude by the above absolute continuity relation, together with the convergence of the unconditioned pair.
\end{proof}

Proposition~\ref{prop:cv_contour_brownien_drift} now easily follows.

\begin{proof}[Proof of Proposition~\ref{prop:cv_contour_brownien_drift}]
In this regime, Lemma~\ref{lem:cv_recentre_brownien} reads in the particular case $N_n \sim (\kappa n/c_{n})^2$:
\begin{equation}\label{eq:cv_brownien_drift}
\left(\frac{c_{n}}{\kappa n} \left(B^{n}_{(\kappa n/c_{n})^2 t}, \widehat{B}^{n}_{(\kappa n/c_{n})^2 t}\right)\right)_{t \ge 0}
\cvloi \left(\mathcal{B}_{t}-\kappa t, \widehat{\mathcal{B}}_{t}-\kappa t\right)_{t \ge 0}.
\end{equation}
The claim is then a consequence of the construction of $W^n$ and $\mathcal{W}^\kappa$ which is continuous in $\mathcal{B}^\kappa$.
\end{proof}

Note that Proposition~\ref{prop:cv_contour_brownien_drift} implies in particular the convergence of $(A^n, R^n)$ after rescaling towards the length of the first excursion of $\mathcal{W}^\kappa$ together with a random time, and this pair has the law $\kappa \overline{\mathbf{n}}^\kappa$ from~\eqref{eq:excursion_biaisee}. Let us give a proof by direct calculations for the reader uncomfortable with excursion theory.

\begin{proposition}[Size of the distinguished tree]\label{prop:volume_arbre_racine}
If $\sqrt{n} \ll c_n \ll n$, then for any $\kappa >0$,
\[\left(\frac{c_n}{\kappa n}\right)^2 \left(A^n, R^n\right) \cvloi (A, R),
\qquad\text{which has law}\qquad
\frac{\kappa}{\sqrt{2 \pi a^3}} \exp\left(- \frac{\kappa^2 a}{2}\right) 1_{0 \le r \le a} \d a \d r.\]
\end{proposition}

\begin{proof}
Let $W_{N,K}$ denote the set of $\pm1$ paths of length $N$ that end by hitting $-K$ for the first time at time $N$, where $N$ and $K$ must have the same parity; by the cycle lemma, its cardinal equals $\#W_{N,K} = \frac{K}{N} \binom{N}{(N+K)/2}$. Then for every $k \in \{0, \dots n-c_{n}\}$, under our biased probability measure on $W_{2n, 2c_{n}}$, 
\[\P(A^n = 2k+1) = \frac{(2k+1) \#W_{2k+1,1} \cdot \#W_{2n-2k-1, 2c_{n}-1}}{\binom{2n}{n+c_{n}}}.\]
Then straightforward calculations involving Stirling's formula lead to
\[\left(\frac{\kappa n}{c_n}\right)^2 \P(A^n = 2\floor{a (\kappa n/c_n)^2}+1) 
\cv \frac{\kappa}{\sqrt{\pi a}} \exp\left(- \kappa^2 a\right)
,\]
for every $a>0$, which implies the convergence in distribution of $A^n$.
The joint convergence of $R^n$ follows since the latter is conditionally given $A^n$ uniformly distributed on $\{0, \dots, A^n-1\}$.
\end{proof}

\section{The mesoscopic scaling limit}
\label{sec:mesoscopic}

In this section we finally prove Theorem~\ref{thm:intermediate} involving the continuum tree-decorated trivalent map $ \mathfrak{F}_{ \mathrm{Unic}}$ or $ \mathfrak{F}_{ \mathrm{Plan}}$ for which we first describe two equivalent constructions. The starting block is the local limit of trivalent maps. 
In the planar case, this is the dual of the well-known UIPT of type 1, denoted by $ \mathbb{UIPT}^\dagger$ in Section~\ref{ssec:defaut_planar}. In the unicellular case, the local limit is the deterministic infinite three-regular tree $\mathbb{A}_3$ which appeared in Section~\ref{ssec:unicellularestimates}. 
In each case, we shall consider a slight modification of those maps obtained by splitting its root edge in two by inserting a vertex in the middle and grafting a dangling edge onto this new vertex in the face adjacent on the right of the root edge. Let us denote by $\mathbb{T}_ \mathrm{Plan}$ and $\mathbb{T}_{ \mathrm{Unic}}$ the resulting maps which thus have a unique vertex of degree $1$, and whose root edge is the oriented edge emanating from this vertex.

Throughout this section, to simplify notation we put
\[\kappa =  \sqrt{\frac{3}{2}}, \qquad \text{so that} \qquad  2 \kappa = \sqrt{6}.\]

\begin{figure}[!ht]\centering
\includegraphics[width=\linewidth]{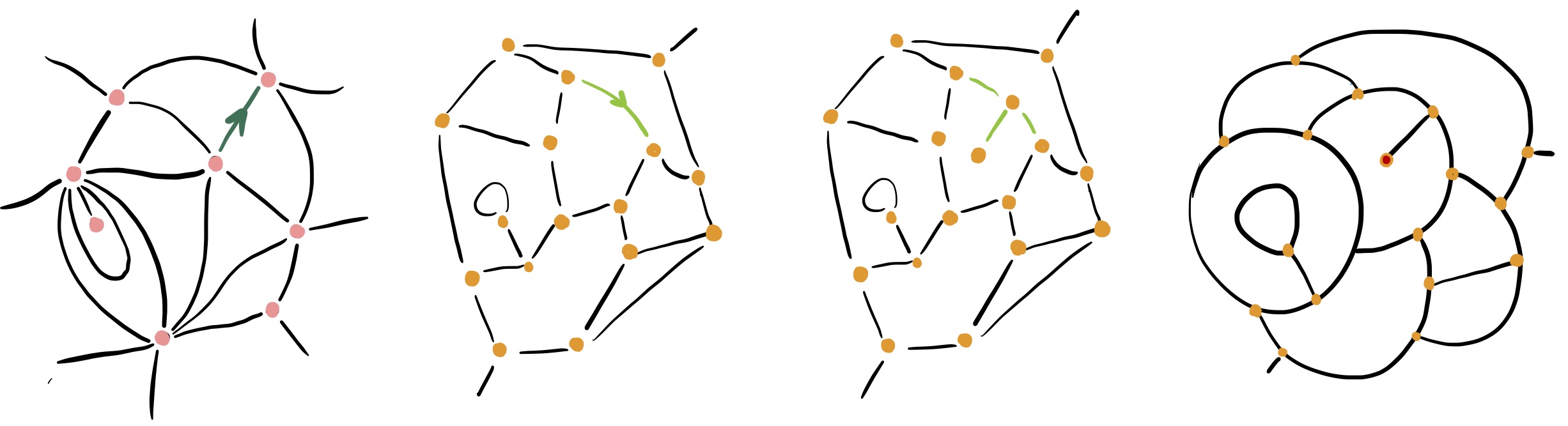}
\caption{The discrete skeleton underlying the construction of $\mathfrak{F}_{ \mathrm{Plan}}$. From left to right: a piece of the Uniform Infinite Planar Triangulation (type 1), its dual $ \mathbb{UIPT}^{\dagger}$, the version $ \mathbb{T}_{ \mathrm{Plan}}$ obtained after the root-transformation and finally the metric graph obtained after dilating each edge independently by an exponential variable of mean $1/\sqrt{6}$.}
\label{fig:squelettediscret}
\end{figure}

\subsection{Construction of the limit}
\label{ssec:construction_limit_meso}

Let $\Mapb_\infty$ denote an infinite, locally finite, map. Let us construct in two equivalent ways a certain metric space $ \mathfrak{F}$ from $\Mapb_\infty$. These constructions applied to $\Mapb_\infty = \mathbb{T}_ \mathrm{Plan}$ and to $\Mapb_\infty =  \mathbb{T}_\mathrm{Unic}$ respectively produce the limits $ \mathfrak{F}_{ \mathrm{Plan}}$ and $ \mathfrak{F}_{ \mathrm{Unic}}$ in Theorem~\ref{thm:intermediate}.
We assume that the reader is familiar with the background on continuum random trees (CRT's).
We shall denote by $\mathbf{e}$ a continuous excursion with duration $\zeta(\mathbf{e})$; it is known (see e.g.~Duquesne \& Le Gall~\cite{DLG05}) that it encodes a CRT $\Tree_{\mathbf{e}}$ by identifying all the pairs of times, say $0 < a < b < \zeta(\mathbf{e})$, which satisfy $\mathbf{e}_a = \mathbf{e}_b = \min_{[a,b]} \mathbf{e}$. We let $\pi_{\mathbf{e}}$ denote the canonical projection $[0, \zeta(\mathbf{e})] \to \Tree_{ \mathbf{e}}$.

\subsubsection{Via bipointed trees surgery}
\label{sssec:bipointed}

Recall from Section~\ref{sec:excursionBMD} the renormalised law $\kappa \overline{ \mathbf{n}}^{\kappa}$ on pairs $( \mathbf{e}, u)$ where $\mathbf{e}$ is a size-biased excursion of a Brownian motion with drift $-\kappa$ and then $u$ is an independent uniform random instant between $0$ and its duration $\zeta(\mathbf{e})$. 
Consider next the law $ \P^{\bullet}$ on bipointed CRT's obtained as the push forward of $\kappa \overline{ \mathbf{n}}^{\kappa}$ by the projection $\pi_{\mathbf{e}}$, or more precisely the map 
\[( \mathbf{e},u) \mapsto ( \Tree_{ \mathbf{e}}, \pi_{\mathbf{e}}(0), \pi_{\mathbf{e}}(u))
.\]
Let us note that by the rerooting property of Brownian CRT's (or more precisely, of Brownian excursions), the two triplets $(\Tree_{ \mathbf{e}}, \pi_{\mathbf{e}}(u), \pi_{\mathbf{e}}(0))$ and $(\Tree_{ \mathbf{e}}, \pi_{\mathbf{e}}(0), \pi_{\mathbf{e}}(u))$ have the same law. We then consider an i.i.d. sample from $\P^{\bullet}$ of bipointed CRT's indexed by the edges of $ \Mapb_\infty$ and we glue these CRT's using their distinguished points according to the adjacency relations of $ \Mapb_\infty$ to get a random locally compact metric space $ \mathfrak{F}$, see Figure~\ref{fig:constructionbis}. Formally, this random compact metric space is obtained by taking the disjoint union of the bipointed CRT's  indexed by the edges of $\Mapb_\infty$  and identifying their distinguished points  according to the adjacency relations of the graph $\Mapb_\infty$, the resulting quotient $ \mathfrak{F}$ is endowed with the quotient metric, see e.g.~\cite[Def.~3.1.12]{BBI01} or the recent paper~\cite{Mug19}.

\begin{figure}[!ht]\centering
\includegraphics[width=\linewidth]{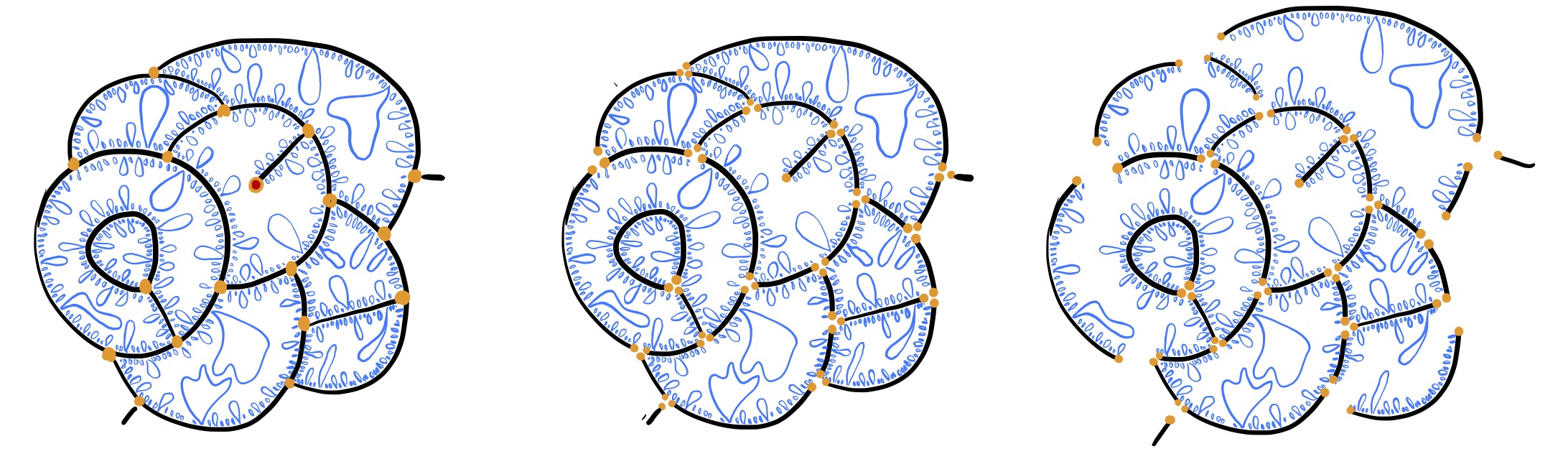}
\caption{The two equivalent constructions of $ \mathfrak{F}$ from the random metric graph $\mathbf{M}_\infty$.}
\label{fig:constructionbis}
\end{figure}

\subsubsection{Via Poissonian theory}
\label{sssec:Poisson_arbres}

Let us give an equivalent construction of $ \mathfrak{F}$ which highlights the connections with the core--kernel decomposition. First, consider the metric (or cable) graph $ \mathbf{M}_\infty$ obtained  by replacing independently each edge of $\Mapb_\infty$ by a compact segment of length distributed according to an exponential law of rate $ 2 \kappa$ (formally defined in the same way as above). This space has a natural Lebesgue measure $\ell$. 
We shall now graft random CRT's on this structure to get our desired space $ \mathfrak{F}$. To do this, consider the infinite measure $\mathbf{n}^{\kappa}$ on the space of pointed compact real trees equipped with the Gromov--Hausdorff topology, obtained by the push forward of the measure $ \mathbf{n}^{\kappa}$ by the application $ \mathbf{e} \mapsto ( \Tree_{ \mathbf{e}}, \pi_{\mathbf{e}}(0))$.
Although this is an infinite measure, the total mass of trees of diameter larger than $  \varepsilon>0$ is finite. We then consider a Poisson cloud on $ \mathbf{M}_\infty$ with intensity
\[2 \d \ell \otimes  \mathbf{n}^{\kappa}\]
and ``graft'' the trees on  $ \mathbf{M}_\infty$ according to the atoms of the measure. 

The fact that the two above constructions are equivalent follows from Bismut's description of the law $ \overline{ \mathbf{n}}^{\kappa}$ recalled in Section~\ref{sec:excursionBMD}. Indeed, once translated in the terms of random trees, this decomposition precisely entails that a random bipointed tree $( \Tree,x,y)$ under $\P^{\bullet}$ can be obtained by first sampling a real segment with a random length with the exponential law with parameter $2 \kappa$ whose endpoints will be the distinguished points of the bipointed tree and then grafting on it a Poisson cloud of trees with intensity $ 2\d \ell \otimes \mathbf{n}^{\kappa}$; the factor $2$ is here to take into account both sides of the segment.

\subsection{Proof of Theorem~\ref{thm:intermediate}}

We are now ready to prove Theorem~\ref{thm:intermediate}. This takes three main steps: first, since the kernel of the maps are almost trivalent, then as discussed in Section~\ref{sec:defauts_bis} it converges locally to the corresponding infinite trivalent map $\Mapb_\infty$. Next, the core is roughly obtained by expanding uniformly at random the edges of the kernel, which translates into i.i.d~exponential random lengths in the limit. Finally, the full map is obtained from the core by grafting trees on the corners, and this forest converges by the results in Section~\ref{sec:contour}.

\begin{proof}[Proof of Theorem~\ref{thm:intermediate}]
\textsc{Step 1: convergence of the kernel.}
By Theorem~\ref{thm:defauts_ker} the kernel of both maps are almost trivalent, in the sense that their defect number are small compared to $\spar_n$ with high probability. On this event, the local limits~\eqref{eq:convUIPT} and~\eqref{eq:convthree-regular} apply. By e.g.~Skorokhod's representation theorem, we shall assume that they hold almost surely and we denote by $\Mapb_n^{\spar_n}$ either the random plane map $\Map_{n}( \spar_{n},0)$ or the unicellular one $\Map_{n}(1, (\spar_{n}-1)/2)$, and by $\Mapb_\infty$ the limit of its kernel, which is $\mathbb{UIPT}^\dagger$ or $\mathbb{A}_3$ respectively. In particular, the kernel is asymptotically locally trivalent.
Fix $r \ge1$, then the ball of radius $r$ (centred at the root vertex) in $\Ker(\Mapb_n^{\spar_n})$ converges almost surely towards that of $\Mapb_\infty$. Since the set of possible such balls is finite, then for every $n$ large enough, the balls coincide and we henceforth assume it is the case. As in Section~\ref{sec:decomposition} we henceforth modify the kernel and $\Mapb_\infty$ by adding a vertex in the middle of its root edge, the corner on the right of the middle vertex is called hereafter the root corner. Note that this is not quite the root transform presented  in the beginning of Section~\ref{sec:mesoscopic}: here we do not graft a dangling leaf on this new vertex.
We let $\Ker_n$ denote the number of edges of the modified kernel.

\textsc{Step 2: convergence of the core.}
Let now $\Core(\Mapb_n^{\spar_n})$ denote the core of the map, with the same modification at the root as in \textsc{step 1} and let $\Core_n$ denote its number of edges.
Recall from Proposition~\ref{prop:loi_core_kernel} that, conditionally on the kernel and the size $\Core_n$, this core is obtained from the kernel by expanding the $\Ker_n$ edges using a uniform random vector of positive integers that sum up to $\Core_n$. Note that the root corner of the kernel is transferred to the core.
Using the representation of such a random vector as i.i.d~geometric random variables conditioned by their sum, where the parameter is arbitrary and can conveniently chosen so the mean matches the average value $\Core_n/\Ker_n$, it is easy to check that for any finite subset of edges of the kernel, $\Ker_n / \Core_n$ times their lengths in the core jointly converge in distribution towards i.i.d.~exponential random variables with unit mean. Alternatively, in the spirit of Section~\ref{sec:decomposition}, for any positive integers $\ell_1, \dots, \ell_k$, the conditional probability that $k$ given edges of the kernel have these lengths equals
\[\frac{\binom{\Core_n-(\ell_1+\dots+\ell_k)-1}{\Ker_n-k-1}}{\binom{\Core_n}{\Ker_n}}.\]
Stirling's approximation then yields a multivariate local form of the convergence to i.i.d.~exponential random variables. 
Fix $r \ge 1$ and let $k_r+1$ denote the number of edges in the ball of radius $r$ of $\Mapb_\infty$, which we assume equals that in $\Ker(\Mapb_n^{\spar_n})$ for $n$ large enough. Let further $L_{n,0}, \dots, L_{n,k_r}$ denote the lengths of these edges in the core. Since $\Ker_n / \Core_n \sim \sqrt{2\Ker_n/n} \sim \sqrt{6\spar_n/n}$ by Theorems~\ref{thm:defauts_ker} and~\ref{thm:aretes_core}, then
\begin{equation}\label{eq:convergence_longueurs_core}
\sqrt{\frac{\spar_n}{n}} \cdot \left(L_{n,0}, \dots, L_{n,k_r}\right) \cvloi (\gamma_0, \dots, \gamma_{k_r}),
\end{equation}
where the $\gamma_i$'s are i.i.d. exponential random variables with mean $1/\sqrt{6}$. Appealing to e.g.~Lebesgue's theorem, this convergence for the conditional law given $\Ker(\Mapb_n^{\spar_n})$ and $\Core_n$ also holds unconditionally, jointly with Theorem~\ref{thm:defauts_ker} and Theorem~\ref{thm:aretes_core}.

\textsc{Step 3: convergence of the trees.}
Next, recall that conditionally on its core, the map $\Mapb_n^{\spar_n}$ is obtained from its core by grafting a rooted plane tree (possibly with a single vertex) onto each corner of the latter. Moreover, the root edge of $\Mapb_n^{\spar_n}$ is either the root edge of the core, or one oriented edge in the tree grafted onto the root corner, hereafter call the ``root tree''.
Let us consider each edge of the kernel, and the corresponding chain of edges in the core, and, except for the root tree, let us group together all the trees grafter on the corners on one side of such a chain (say from one extremity to the other one) and then those on the other side. The root tree is canonically placed first.
Then this forest, together with the root edge, is sampled uniformly at random amongst all possibilities and it is coded by the first-passage path $W^n$ studied in Section~\ref{sec:contour}, with one distinguished time $R^n$ smaller than the first hitting time of $-1$. 

Then a direct consequence of Proposition~\ref{prop:cv_contour_brownien_drift} and Bismut's decomposition is that, conditionally on the kernel and the core of the map, $\sqrt{\spar_n/n}$ times the root tree and its mark converge to a bipointed Brownian tree with law $ \P^{\bullet}$ as defined in Section~\ref{sssec:bipointed}. Moreover for every $r \ge 1$, jointly with this convergence and that~\eqref{eq:convergence_longueurs_core} of the lengths of the chains in the core replacing the edges of the ball of radius $r$ in the kernel, the forests of the trees grafted on both sides of these chains jointly converge after the same rescaling by $\sqrt{\spar_n/n}$ to independent forests coded by Brownian motions with drift $\mathcal{B}^\kappa$ killed when first reaching level $-\gamma_i$ respectively, where we recall that we take $\kappa = \sqrt{3/2}$. Since these $\gamma_i$'s have the exponential law with rate $\sqrt{6} = 2\kappa$, then Bismut's decomposition (recall the discussion closing Section~\ref{sssec:Poisson_arbres}) entails that the bipointed trees obtained by grafting all the trees, except the root tree, in the corners of each chain in the core converge after rescaling by $\sqrt{\spar_n/n}$ to i.i.d~bipointed Brownian trees with law $ \P^{\bullet}$. 
Recall from Section~\ref{ssec:construction_limit_meso} and especially Figure~\ref{fig:squelettediscret} that there when constructing the limit $\mathfrak{F}$, we not only added a middle vertex on the root edge of $\Mapb_\infty$, but also attached to it a dandling leaf on the root corner and this edge was eventually replaced by a bipointed CRT in $\mathfrak{F}$. This CRT is the limit of the root tree and its mark here and again Bismut's decomposition shows the equivalence of the two points of views.

\textsc{Step 4: convergence of the map.}
Again, the previous conditional invariance principle is extended unconditionally by Lebesgue's theorem so we conclude that for any $r\ge 1$, the subset $\mathbb{B}_{n,r}^{\spar_n}$ of $\Mapb_n^{\spar_n}$ obtained by taking the ball of radius $r$ of its kernel and replacing its edges by their corresponding bipointed trees in $\Mapb_n^{\spar_n}$ converges in distribution, once rescaled by $\sqrt{\spar_n/n}$ to the ball of radius $r$ in $\Mapb_\infty$ where each edges is replaced by i.i.d. bipointed CRT's with law $\P^{\bullet}$, with the extra twist for the root. 
In order to conclude with the local Gromov--Hausdorff convergence of the map $\Mapb_n^{\spar_n}$ to $\mathfrak{F}$, it still remains to argue that for any fixed real value $R>0$, the ball of radius $R \sqrt{n/\spar_n}$ of this map is contained in $\mathbb{B}_{n,r}^{\spar_n}$ for some $r$. Indeed, it could a priori happen that, thanks to very short lengths, points which lie at a large graph distance from the root in the kernel get very close to the root in the core and then in the map.
Now recall that with high probability, the kernel of $\Mapb_n^{\spar_n}$ is locally trivalent so for every $r \ge 1$, its ball of radius $r$ contains at most $3^{r}$ distinct non self-intersecting paths. Then a crude large deviation argument shows that, if $\delta>0$ is small enough, then a.s.~for all sufficiently large $r$'s, none of these rescaled paths can have a total random length smaller than $\delta r$. Consequently with high probability, the ball of radius $R \sqrt{n/\spar_n}$ in $\Mapb_n^{\spar_n}$ is indeed entirely contained in $\mathbb{B}_{n,r}^{\spar_n}$ for $r \ge \delta^{-1} R$.
\end{proof}

\subsection{The tree regimes}
\label{ssec:limites_abres}

Let us end this section with the behaviour of the random map $\Mapb_n^{\spar_n}$ when seen at a smaller scale than $\sqrt{n/\spar_n}$, which complements Theorem~\ref{thm:intermediate}. As we have seen in the previous subsection, the tree grafted onto the core which carries the root edge of $\Mapb_n^{\spar_n}$ grows like $\sqrt{n/\spar_n}$ and so does the distance between the root vertex of the map and the core. Therefore, if one looks in a ball centred at the root vertex with a much smaller radius, then one does not escape this tree so we expect the maps to converge to trees at such scales. Let us describe more precisely these limits before stating the result.

Analogously to the compact Brownian CRT coded by a Brownian excursion, one can consider a Self-Similar Continuum Random Tree $\Tree_\infty$, coded by a two sided Brownian motion $(\mathcal{B}_t)_{t \in \R}$, i.e.~a random path such that $(\mathcal{B}_t)_{t \ge 0}$ and $(\mathcal{B}_{-t})_{t \ge 0}$ are two independent standard Brownian motions, see~\cite[Section~2.5]{Ald91c}. This tree is naturally pointed at the image of $0$ and it possesses a unique infinite line, corresponding to the first hitting time of a negative level by both Brownian motions; the excursions above their infimum of each of these paths code the subtrees grafted along this spine, on each side. Another, ``upward'', way of constructing $\Tree_\infty$ is to take instead two independent three-dimensional Bessel processes. The fact that this defines the same object in law comes from the so-called Pitman transform, which shows that such a Bessel process has the same law as $(\mathcal{B}_t - 2\inf_{[0,t]} \mathcal{B})_{t \ge 0}$, and the fact that the tree coded by the latter is the same as that coded by $(\mathcal{B}_t)_{t \ge 0}$ since one can easily that the corresponding pseudo-distances are equal. See also~\cite[Section~7.7.6]{Pit06}.

Finally, a discrete analogue, the Uniform Infinite Random Plane Tree $\mathbb{A}_\infty$, can be described as the discrete tree coded similarly by a two-sided simple random walk, or equivalently two independent such random walks conditioned to stay nonnegative. This infinite tree appears as local limit of large uniform random plane trees; it has one end and is sometimes referred to as Kesten's tree conditioned to survive (with the critical geometric distribution), see e.g.~\cite[Section~5]{Jan12}.

\begin{proposition}\label{prop:limites_arbres}
Let $\spar_n$ satisfy~\eqref{eq:standing}.
\begin{enumerate}
\item Both $\Map_{n}( \spar_{n},0)$ and $\Map_{n}(1, (\spar_{n}-1)/2)$ converge in distribution to $\mathbb{A}_\infty$ for the local topology.

\item For any sequence $a_n \to \infty$ such that $a_n = o(\sqrt{n/\spar_n})$, the two convergences in distribution
\[a_n^{-1}  \Map_{n}( \spar_{n},0) \cvloi \Tree_\infty
\qquad \text{and} \qquad 
a_n^{-1}  \Map_{n}\left(1,\frac{\spar_{n}-1}{2} \right) \cvloi \Tree_\infty\]
hold in the local pointed Gromov--Hausdorff topology.
\end{enumerate}
\end{proposition}

See~\cite{KM21a} for related results on local limits of planar \emph{graphs}.
Let us mention that $\Tree_\infty$ also appears in the scaling limit of uniform random quadrangulations with $n$ internal faces and with a boundary with length much larger than $\sqrt{n}$, see~\cite[Theorem~3.4]{BMR19}.

\begin{proof}
For a short proof, one can note from the proof of Proposition~\ref{prop:volume_arbre_racine} that, conditional on the number of edges of the tree grafted onto the core which carries the root edge of the map, this tree has the uniform distribution on plane trees with such a size, and further the oriented edge is independently sampled uniformly at random. Then re-root this tree at this oriented edge: the resulting tree has again the uniform distribution and the latter is known to converge when its size tends to infinity, see~\cite[Theorem~7.1]{Jan12} for the local convergence and~\cite[Section~2.5]{Ald91c}, with the formalism from~\cite{DLG05} for the local Gromov--Hausdorff one.
\end{proof}

The claim can alternatively be proved along the same lines as previously, we keep the same notation.
Indeed, instead of Equation~\eqref{eq:cv_brownien_drift} which was used previously, Lemma~\ref{lem:cv_recentre_brownien} shows that if $\sqrt{n} \ll c_n \ll n$ and if $N_{n} \to \infty$ is such that $N_{n} = o((n/c_{n})^2)$, then the drift disappears and we simply obtain,
\[\left(\frac{1}{\sqrt{N_{n}}} \left(B^{n}_{N_{n}t}, \widehat{B}^{n}_{N_{n}t}\right)\right)_{t \ge 0}
\cvloi \left(\mathcal{B}_{t}, \widehat{\mathcal{B}}_{t}\right)_{t \ge 0},\]
where $\mathcal{B}$ and $-\widehat{\mathcal{B}}$ are independent standard Brownian motions. 
For $r>0$ fixed, 
the concatenation of the these paths stoped when first reaching level $r$ encodes a tree in which $0$ is a distinguished time. 
The convergence of paths then implies the convergence of the corresponding bipointed trees; in particular, the ball of radius $r$ in $N_n^{-1/2} \Mapb_n^{\spar_n}$ converges to that of the image of $0$ in the preceding excursion, which is the ball of radius $r$ in $\Tree_\infty$.
We then applies this result to the random number of edges of the core, which by Theorem~\ref{thm:aretes_core} grows like $\sqrt{n\spar_n}$, so $N_n$ can be any sequence with $N_n = o(n/\spar_n)$.

Similarly, the proof of Lemma~\ref{lem:cv_recentre_brownien} is easily adapted to show that 
for any fixed $N \in \N$,
\[\big(B^{n}_{i}, \widehat{B}^{n}_{i}\big)_{0 \le i \le N} \cvloi \big(S_{i}, \widehat{S}_{i}\big)_{0 \le i \le N},\]
in $\R^{N+1}$, where $S$ and $\widehat{S}$ are two independent simple random walks. 
This similarly implies the convergence in distribution of the map in the local topology towards $\mathbb{A}_\infty$.

\section{Comments and questions}
\label{sec:comments}

Let us finish this paper by raising a few problems and open questions.

\subsection{Short cycles and diameter in the unicellular case}

As mentioned in the introduction, Janson \& Louf~\cite{JL21b} very recently proved that the statistics of the length of short cycles in $ \Map_{n}( 0,  \gen_{n})$ converge when $ 1 \ll \gen_n \ll n$ after normalisation by $ \sqrt{n/(12  \gen_{n})}$ towards an inhomogeneous Poisson process on $ \mathbb{R}_{+}$ with intensity
\begin{equation}\label{eq:cosht}
\frac{\cosh (t)-1}{t},
\end{equation}
and this matches the statistics of the lengths of short non-contractible curves in Weil--Petersson random surfaces~\cite{MP19}. Let us shed some light on these results using ours. We do not however claim to give a full proof. Heuristically we saw that $ \sqrt{12  \gen_{n}/n} \cdot \mathrm{Core}( \Map_{n}( 1,  \gen_n))$ is given by first taking an essentially unicellular trivalent map whose edges have been replaced by independent real segments of length distributed according to an exponential law of mean $1$.  It is classical that the statistics of cycles of length $k\geq 1$ in a random three-regular multi-graph (where loops and multiple edges are allowed) are given by independent Poisson variables with parameters
\[1,\enskip1,\enskip\frac{8}{6},\enskip \dots,\enskip \frac{2^k}{2k},\enskip \dots\]
see~\cite[Theorem~2.5]{Wor99}. Assuming this result still holds in the case of unicellular and essentially trivalent maps, we can guess that the statistics of the length of the short cycles in the rescaled map $ \sqrt{12  \gen_{n}/n} \cdot \mathrm{Core}( \Map_{n}( 1,  \gen_n))$  are obtained by taking the image of the above Poisson process on the number of discrete cycles, after replacing each cycle of length $k$ by an independent sum of $k$ random exponential variables. It is easy to check that the resulting point process on $ \mathbb{R}_+$ is Poisson with intensity given by~\eqref{eq:cosht}. Turning this sketch into a rigorous proof would open another path towards the result of~\cite{JL21b} and get access to more global quantities of sparse unicellular maps such as the diameter.

\subsection{Global limits in the planar case}
\label{ssec:TBM}

Let us here focus on the random plane maps  $\Map_n( \face_n, 0)$. Theorem~\ref{thm:intermediate} and Proposition~\ref{prop:limites_arbres} study their asymptotic behaviour at the scales $\sqrt{n/\face_n}$ and smaller.
In another direction, one can be interested in their asymptotic geometry at larger scales.

At least when $\face_n = o(n^{1/3})$, we know from Theorem~\ref{thm:defauts_ker} that $\Ker(\Map_n( \face_n, 0))$ is trivalent with probability tending to $1$, so by~\cite[Corollary~23]{CLG19}, once rescaled by a factor of order $\face_n^{1/4}$, it converges to the Brownian sphere. Combined with our previous argument (e.g.~in the proof of Theorem~\ref{thm:intermediate}), this indicates that $\face_n^{1/4} \sqrt{n/ \face_n} = (n^2/\face_n)^{1/4}$ is the correct scale of the core and we believe that it also converges to the Brownian sphere by arguments similar to~\cite{CLG19}. Finally, the original map $\Map_n( \face_n, 0)$ is obtained by grafting trees on the core, with $n$ edges in total, so the maximal diameter of such a tree grows like
$
\sqrt{n/\face_n} = o((n^2/\face_n)^{1/4})$ and therefore the rescaled map and its core should be close to each other. 
We refrain to make this precise here for we believe that this holds in a more general setting.

\begin{conjecture}\label{conjecture}
If $\face_n \to \infty$ and $n^{-1} \face_n \to 0$, then the rescaled maps
\[\sqrt{6} \cdot \frac{3^{1/4}}{1+2\sqrt{3}} \cdot \left(\frac{\face_n}{n^2}\right)^{1/4} \Map_n( \face_n, 0)\]
converge in distribution to the Brownian sphere in the Gromov--Hausdorff topology.
\end{conjecture}

Similarly, for any sequence $\sigma_n$ such that $\sqrt{n/\face_n} \ll \sigma_n \ll (n^2/\face_n)^{1/4}$, we expect the rescaled random map $\sigma_n^{-1} \Map_n(\face_{n},0)$ to converge in distribution for the pointed local Gromov--Hausdorff topology towards the Brownian \emph{plane}, which is a non-compact analogue of the Brownian sphere introduced in~\cite{CLG14}.

The first step towards a proof of Conjecture~\ref{conjecture} would be to prove the convergence of random trivalent maps with a small defect number compared to the number of edges to the Brownian sphere, once rescaled by the fourth root of the number of edges. This would complement the local point of view of~\cite{Bud21}. 
Let us mention that, because we were especially interested in the mesoscopic behaviour of the maps, we assumed throughout this work that $\face_n = o(n)$, whereas in~\cite{KM21b}, convergence of uniformly chosen \emph{bipartite} maps to the Brownian sphere is shown whenever $\face_n, n-\face_n \to \infty$. More generally, establishing the scaling and local limits of planar maps with prescribed degrees is still open in the non-bipartite case. See~\cite{Mar19} for the scaling limit point of view and~\cite{BL20} for the local limit point of view, both in the bipartite case. Let us mention that the case of local limit of Boltzmann non-bipartite maps  is treated in~\cite{Ste18}.

\end{document}